\documentclass[11pt]{amsart}
\newcommand{\ds}{\displaystyle}
\newcommand{\ben}{\begin{enumerate}}
\newcommand{\een}{\end{enumerate}}
\newcommand{\eq}[2][label]{\begin{equation}\label{#1}#2\end{equation}}
\newcommand{\av}[2]{\langle #1\rangle_{_{\scriptstyle #2}}}
\newcommand{\ave}[1]{\langle #1\rangle}
\newcommand{\ve}{\varepsilon}
\newcommand{\rtde}{\sqrt{\delta^2-\ve^2}}
\newcommand{\rtdet}{\sqrt{\delta^2-\ve^2/2-\theta_*^2}}
\newcommand{\Oe}{\Omega_\ve}
\newcommand{\Od}{\Omega_\delta}
\usepackage{amsfonts}
\usepackage{amssymb}
\usepackage{amsmath}
\usepackage{graphicx}
\newcommand{\bel}[1]{\boldsymbol{#1}}
\newcommand{\BMO}{{\rm{BMO}}}
\newcommand{\df}{\stackrel{\mathrm{def}}{=}}
\newcommand{\half}{\frac12}

\newtheorem{thmc}{Theorem}

\newtheorem{thmd}{Theorem}

\newtheorem{lemc}{Lemma}

\newtheorem{lemd}{Lemma}

\newtheorem{prop}{Proposition}

\newtheorem{obs}{Observation}

\newtheorem{propm}{Proposition}

\newtheorem{obsm}{Observation}

\newtheorem*{theorem*}{Theorem}{\bf}{\it}
\newtheorem*{proposition*}{Proposition}{\bf}{\it}
\newtheorem*{observation*}{Observation}{\bf}{\it} 
\newtheorem*{lemma*}{Lemma}{\bf}{\it}

\theoremstyle{definition}

\theoremstyle{remark}

\numberwithin{equation}{section}

\allowdisplaybreaks
\begin{document}

\title{Sharp results in the integral-form John--Nirenberg inequality}

\author{L. Slavin} 
\address{University of Cincinnati}
\email{leonid.slavin@uc.edu}

\author{V. Vasyunin}
\address{St. Petersburg Department of the V.~A.~Steklov Mathematical Institute, RAS}
\email{vasyunin@pdmi.ras.ru}
\thanks{V. Vasyunin's research supported in part by RFBR (grant no. 05-01-00925)}



\date{May 7, 2008}

\begin{abstract}
We consider the strong form of the John-Nirenberg inequality for the $L^2\!$-based $\BMO.$ We construct explicit Bellman functions for the inequality in the continuous and dyadic settings and obtain the sharp constant as well as the precise bound on the inequality's range of validity, both previously unknown. The results for the two cases are substantially different. The paper not only gives another instance in the short list of such explicit calculations, but also presents the Bellman function method as a sequence of clear steps, adaptable to a wide variety of applications.
\end{abstract}

\maketitle
\makeatletter
\makeatother
{\small
\tableofcontents
}
\section{Introduction}
\label{intro} 
In this paper, we are dealing with the space $\BMO$ that first appeared in the
classical paper \cite{jn}. A crucial property of elements of $\BMO,$ the exponential decay of
their distribution function, was also established in that paper; it is now known as the
weak-form John--Nirenberg inequality.

For an interval $I,$ and a real-valued function $\varphi\in L^1(I),$ let $\av{\varphi}I$ be
the average of $\varphi$ over $I,$ $\av{\varphi}I=\frac1{|I|}\int_I\varphi.$ For $1\le
p<\infty,$ let \eq[i1]{ \BMO(I)=\left\{\varphi\in L^1(I): \av{|\varphi-\av{\varphi}J|^p}J\le
C^p<\infty,\; \forall~\text{interval}~J\subset I\right\} } with the best (smallest) such $C$
being the corresponding norm of $\varphi.$ The classical definition of John and Nirenberg
uses $p=1;$ it is known that the norms for different $p\!$'s are equivalent. For every
$\varphi\in\BMO(I)$ and every $\lambda\in\mathbb{R}$ one has
\begin{theorem*}[John, Nirenberg; weak form]
\eq[i1.5]{
\frac1{|I|}|\{s\in I:~\varphi-\av{\varphi}I>\lambda\}|\le c_1e^{-c_2\lambda/\|\varphi\|_{\BMO(I)}}.
}
\end{theorem*}
$\BMO$ plays a major role in modern analysis (in particular, because it is dual to the Hardy
space $H^1$ \cite{fefferman}). In addition, inequality \eqref{i1.5} can be viewed as an
accurate characterization of unbounded $\BMO$ functions. It is thus of great interest to
determine sharp constants $c_1$ and $c_2.$ For the classical case $p=1,$ Korenvoskii
\cite{korenovsky} established the exact value $c_2=2/e.$ Inequality \eqref{i1} can be
integrated to produce an equivalent statement. For $\ve\ge0,$ let
$$
\BMO_\ve(I)=\{\varphi\in\BMO(I):\|\varphi\|\le\ve\}.
$$
Then we have
\begin{theorem*}[John, Nirenberg; integral form]
\label{jn}
There exists $\ve_0>0$ such that for every $0\le\ve<\ve_0$ there is
$C(\ve)>0$ such that for any function $\varphi \in \BMO_{\ve}(I),$
\eq[i2]{
\av{e^\varphi}I\le C(\ve)e^{\av{\varphi}I}.
}
\end{theorem*}
This paper has two main objectives: the first one is to establish, for the case $p=2,$ the
sharp values for $\ve_0$ and $C(\ve)$ in \eqref{i2}. We accomplish this for the continuous
$\BMO$ defined above as well as its dyadic analog $\BMO^d,$ for which every subinterval $J$
of $I$ in definition \eqref{i1} is an element of the dyadic lattice rooted in $I.$

The second objective is to showcase the tool that is at the center of the proofs. It is the
Bellman function method, a powerful harmonic analysis technique developed in the past 12
years. In the important paper \cite{burkholder}, Burkholder found what can now be understood
as the first explicit harmonic analysis Bellman functions. However, his language was
different and the method did not appear in its present form until 1995, when a two-weight
martingale transform was handled in \cite{haar} (later published as \cite{haar1}). In the big
paper \cite{nt}, the authors define many Bellman functions, as a matter of both developing
the method and solving several important problems. Many results, using different variants of
the technique, have followed but until \cite{vasyunin} was published in 2003, none had found
their Bellman functions explicitly, instead relying on Bellman-type arguments, when one uses
a substitute function with similar size and concavity properties.

The list of explicit Bellman functions is still very short. Besides \cite{vasyunin}, we note
the papers \cite{melas,vv,svol}; several others are in the works. While the present paper
gives one of the earliest known such computations (see \cite{s,vasyunin-prep}), it has taken
time to bring it to print. Finding the corresponding Bellman function exactly will always
yield sharp results for an inequality, but this paper also has methodical value: it is our
hope that it will further a new paradigm in Bellman investigations, help bring about a new
pure-Bellman template. To describe it briefly, upon choosing the Bellman variables and
setting up the corresponding extremal problem, one is to establish the finite-difference
inequality(ies) codifying the concavity (convexity) of the Bellman function along the
trajectories defined by the choice of variables. The inequality then is rephrased as a set of
partial differential conditions, which are ``sharpened'' to become what we now call ``the
Bellman PDE.'' Using homogeneity inherent in the problem, one reduces the order of the PDE
and finds a solution, a ``candidate'' Bellman function. Then, one proves that the candidate
is indeed the true Bellman function, using a dyadic-type induction on scales in one direction
and finding an extremal function to establish the other. We follow this template in both,
continuous and dyadic, cases.

Surprisingly, the solution of that PDE turns out to be the Bellman function for the {\it
continuous} John--Nirenberg setup, and that takes a substantial amount of work to show. We
then solve the dyadic case, using the continuous solution as a starting point. The results
for the two cases turn out to be drastically different.

As the name suggests, the method has its origins in stochastic optimal control. We refer the
reader to papers \cite{vol,vol1} where the connection between the two incarnations of the
method is explored. In fact, it was an early version of \cite{vol} where we first saw a
Bellman setup for a dyadic version of inequality \eqref{i2}. The authors then stated a formal
PDE for the problem and found a majorant of its solution, in effect establishing the dyadic
inequality with some suboptimal values. Each of us, independently, solved the PDE exactly,
and we then pooled our efforts to proceed from this formal solution to the rigorous proof of
our theorems. We would especially like to acknowledge the help of A.~Volberg who formulated
the problem to each of us and brought us together.
\section{The Bellman setup}
\label{setup}
We use definition \eqref{i1} with $p=2.$ The main reason is that it can then be
rewritten as
$$
\BMO(I)=\left\{\varphi\in L^1(I):
\av{\varphi^2}J-\av{\varphi}J^2\le C^2,\;\forall~\text{interval}~J\subset I\right\}
$$
with the norm
$$
\|\varphi\|_{\BMO(I)}=\left(\sup_{J\subset I}\left\{\av{\varphi^2}J-\av{\varphi}J^2\right\}\right)^{1/2},
$$
with the appropriate modifications for the dyadic space $\BMO^d.$ This rewriting greatly
facilitates the description of the problem in terms of Bellman variables, as shown below.

As mentioned above, by $\BMO_\ve(I)$ and $\BMO^d_\ve(I)$ we denote the $\ve\!$-ball (the ball
of radius $\ve$ centered at $0$) in the corresponding space. With every such ball and the set
of all subintervals $J\subset I$ we associate the domain $\Oe=\{x=(x_1,x_2)\colon x_1\in
\mathbb{R},\; x_1^2\le x_2\le x_1^2+\ve^2\},$ as follows 
\eq[var]{ 
(\varphi, J)\longmapsto
\left(\av{\varphi}J,\av{\varphi^2}J\right). 
} 
This map is well-defined because
$\av{\varphi}J^2\le \av{\varphi^2}J$ (Cauchy inequality) and $\varphi\in \BMO_{\ve}(I)$
$(\BMO^d_{\ve}(I)).$ On $\Oe$ we define the following Bellman functions 
\eq[t6]{
\bel{B}^+_{\ve}(x)=\sup_{\varphi\in \BMO_{\ve}(I)}\left\{\av{e^\varphi}I:~ \av{\varphi}I=x_1,
\av{\varphi^2}I=x_2\right\}, 
} 
\eq[t6.1]{ 
\bel{B}^-_{\ve}(x)=\inf_{\varphi\in
\BMO_{\ve}(I)}\left\{\av{e^\varphi}I:~ \av{\varphi}I=x_1, \av{\varphi^2}I=x_2\right\}, 
}
\eq[t7]{ 
\bel{B}^{d+}_{\ve}(x)=\sup_{\varphi\in \BMO^d_{\ve}(I)}\left\{\av{e^\varphi}I:~
\av{\varphi}I=x_1, \av{\varphi^2}I=x_2\right\}, 
} 
\eq[t7.1]{
\bel{B}^{d-}_{\ve}(x)=\inf_{\varphi\in \BMO^d_{\ve}(I)}\left\{\av{e^\varphi}I:~
\av{\varphi}I=x_1, \av{\varphi^2}I=x_2\right\}. 
}
Observe that these functions do not depend
on $I.$ The functions with ``$+$'' give the exact upper bound on $\av{e^\varphi}I$ (and so the
sharp John--Nirenberg inequality), while the ones with ``$-$'' give the lower bound. While the
overall lower bound (over all $x$) is well-known ($\av{e^\varphi}I\ge e^{\av{\varphi}I},$ by
Jensen's inequality), the lower Bellman functions give nontrivial results for each particular
choice of $x.$ In addition, they arise naturally in the process of solving the Bellman PDE.

Until now, a typical Bellman function proof would first establish a dyadic result and then
try to come up with a continuous analog. A remarkable feature of our result is that we first
find a family of ``continuous'' Bellman functions and then choose appropriate members of
that family to deal with the dyadic case.
\section{Main results}
\label{main} 
Throughout the paper, we will mark results about the continuous case with index
``c'' and their dyadic analogs with index ``d.'' 
\begin{thmc}
\label{t1c}
Let $\ve_0=1.$ For every $0\le\ve<\ve_0,$ let
\eq[t9]{
C(\ve)=\frac{e^{-\ve}}{1-\ve}.
}
Then\textup, for any $\varphi\in \BMO_{\ve}(I),$
\eq[t8]{
\av{e^\varphi}I\le C(\ve)e^{\av{\varphi}I}.
}
Moreover\textup, $\ve_0$ and $C(\ve)$ are sharp.
\end{thmc}
\begin{thmd}
\label{t1d} 
Let $\ve^d_0=\sqrt2\log 2.$ For every $0\le\ve<\ve^d_0,$ let
\eq[t9.5]{
C^d(\ve)=\frac{e^{-\frac\ve{\sqrt2}}}{2-e^{\frac\ve{\sqrt2}}}, }
Then\textup, for any
$\varphi\in \BMO^d_{\ve}(I),$
\eq[t10]{
\av{e^\varphi}I\le C^d(\ve)e^{\av{\varphi}I}. }
Moreover\textup, $\ve^d_0$ and $C^d(\ve)$ are sharp.
\end{thmd}
Throughout our presentation we will repeatedly use the following very simple fact.
\begin{prop}
\label{pr1}
If $0\le t_1\le t_2,$ then $(1-t_1)e^{t_1}\ge(1-t_2)e^{t_2}$ and
$(1+t_1)e^{-t_1}\ge(1+t_2)e^{-t_2}.$
\end{prop}
\begin{proof}
Since $\frac d{dt}\left((1-t)e^t\right)=-te^t$ and $\frac
d{dt}\left((1+t)e^{-t}\right)=-te^{-t},$ the functions $t\mapsto(1-t)e^t$ and
$t\mapsto(1+t)e^{-t}$ are decreasing for $t>0.$
\end{proof}
Theorems \ref{t1c} and \ref{t1d} are immediate consequences of Proposition~\ref{pr1} and the following results for
the Bellman functions~\eqref{t6}-\eqref{t7.1}. Let
\eq[t12]{
\begin{array}{l}
\ds B^+_{\delta}(x)=\frac{1-\sqrt{\delta^2+x_1^2-x_2}}{1-\delta}\exp\left(x_1+\sqrt{\delta^2+x_1^2-x_2}-\delta\right),\\
\ds B^-_{\delta}(x)=\frac{1+\sqrt{\delta^2+x_1^2-x_2}}{1+\delta}\exp\left(x_1-\sqrt{\delta^2+x_1^2-x_2}+\delta\right).
\end{array}
}
\begin{thmc}
\label{t2c}
If $0\le\ve<1,$ then
$$
\bel{B}^+_{\ve}(x)=B^+_{\ve}(x);
$$
if
$\ve\ge 1,$ then
$$
\bel{B}^+_{\ve}(x)=
\begin{cases}
e^{x_1}&\hbox{if}~x_2=x_1^2\\
+\infty&\hbox{if}~x_2>x_1^2.
\end{cases}
$$
In addition\textup,
$$
\bel{B}^-_{\ve}(x)=B^-_{\ve}(x),\qquad\forall\ve\ge0.
$$
\end{thmc}
\begin{thmd}
\label{t2d} If $0\le\ve<\sqrt2\log 2,$ then 
$$
\bel{B}^{d+}_\ve(x)=B^+_{\delta^+(\ve)}(x),
$$
where $\delta=\delta^+(\ve)$ is the unique
solution of the equation
\eq[t11]{
(1-\rtde)e^{\rtde}\left(2-e^{\ve/\sqrt2}\right)-(1-\delta)e^{\delta-\ve/\sqrt2}=0; 
}
if
$\ve\ge \sqrt2\log 2,$ then
$$
\bel{B}^{d+}_{\ve}(x)=
\begin{cases}
e^{x_1}&\hbox{if}~x_2=x_1^2\\
+\infty&\hbox{if}~x_2>x_1^2.
\end{cases}
$$
In addition, 
$$
\bel{B}_\ve^{d-}(x)=B^-_{\delta^-(\ve)}(x),\qquad\forall \ve\ge0, 
$$
where $\delta=\delta^-(\ve)$ is the unique solution of the equation 
\eq[t11d]{
(1+\rtde)e^{-\rtde}\left(2-e^{-\ve/\sqrt2}\right)-(1+\delta)e^{-\delta+\ve/\sqrt2}=0,~~~\ve\ge0.
}
\end{thmd}
Theorems~\ref{t1c} and \ref{t1d} immediately follow from the Theorems~\ref{t2c} and \ref{t2d}, respectively. Indeed,
Proposition~\ref{pr1} implies that $B^+_\ve$ and $B^+_\delta$ assume their maxima on the upper
boundary of $\Oe$, i.~e. when $x_2=x_1^2+\ve^2;$ so we have
$$
B^+_\ve(x)\le \frac{\ e^{-\ve}}{1-\ve}\;e^{x_1}
$$
and
$$
B^+_{\delta^-(\ve)}(x)\le\frac{C(\delta)}{C(\rtde\;)}\;e^{x_1}=
\frac{e^{-\frac\ve{\sqrt2}}}{2-e^{\frac\ve{\sqrt2}}}\;e^{x_1}
$$
giving~\eqref{t8} and~\eqref{t10} with the sharp constants~\eqref{t9} and~\eqref{t9.5}.

We will first consider the continuous case and then the dyadic one.
\section{The continuous case}
We split the proof of Theorem 2c into two parts.
\begin{lemc}
\label{l1c}
For every $x\in\Oe,$
\eq[1]{
\bel{B}^+_{\ve}(x)\ge
B^+_{\ve}(x);\qquad\bel{B}^-_{\ve}(x)\le B^-_{\ve}(x),
}
where $0<\ve<1$ for $B^+$ and
$\ve>0$ for $B^-.$
\end{lemc}
We prove each of inequalities \eqref{1} by explicitly finding a function $\varphi$ for every point $x\in
\Oe$ such that $\left(\av{\varphi}I,\av{\varphi^2}I\right)=(x_1,x_2)$ and
$$
\av{e^\varphi}I=B_{\ve}(x_1,x_2).
$$
Here $B_\ve$ stands for $B^+_\ve$ or $B^-_\ve$ and the result will then follow from the definition of $\bel{B}^\pm_\ve.$
\begin{proof}
Since $x_2=x_1^2$ occurs if and only if $\varphi=x_1=\text{const},$ it is clear
that $\bel{B}^\pm_0(x)=B^\pm_0(x)=e^{x_1}.$ So we only need to consider $\ve>0.$

Take $I=[0,1],~a\in(0,1],~b\in\mathbb{R},~\gamma\in\mathbb{R}\backslash\{0\}.$ Let
$$
\varphi_{a,b,\gamma}(t)=
\begin{cases}
\gamma\log\frac{a}t+b~&\hbox{for}~0\le t\le a
\\
b~&\hbox{for}~a\le t\le1.
\end{cases}
$$
Let us calculate the $\BMO$ norm of $\varphi_{a,b,\gamma}.$ To simplify calculations, let
$l(t)=\log(a/t)$ and observe that
$$
\int \left(\gamma l(t)+b\right)\,dt=(\gamma+b)t+\gamma tl(t)+C
$$
and
$$
\int \left(\gamma l(t)+b\right)^2\,dt=(2\gamma^2+2\gamma b+b^2)t+\gamma^2
tl^2(t)+2\gamma(b+\gamma)tl(t)+C.
$$
Take an interval $[c,d]\subset I.$ We have the following trichotomy \ben \item $0\le c<d\le
a\le1.$ In this case
$$
\av{\varphi}{[c,d]}=\gamma+b+\gamma\frac{dl(d)-cl(c)}{d-c}
$$
and
$$
\av{\varphi^2}{[c,d]}=2\gamma^2+2\gamma b+b^2+2\gamma(b+\gamma)\frac{dl(d)-cl(c)}{d-c}
+\gamma^2\frac{dl^2(d)-cl^2(c)}{d-c}.
$$
Therefore,
\begin{align*}
\av{\varphi^2}{[c,d]}-\av{\varphi}{[c,d]}^2&=\gamma^2+\frac{\gamma^2}{(d-c)^2}
\left[(dl^2(d)-cl^2(c))(d-c)-(dl(d)-cl(c))^2\right]\\
&=\gamma^2-\frac{\gamma^2cd}{(d-c)^2}\left[l(d)-l(c)\right]^2\le\gamma^2.
\end{align*}
\item
$0\le c\le a\le d\le1.$ In this case
$$
\av{\varphi}{[c,d]}=\frac{-\gamma cl(c)+(b+\gamma)(a-c)+b(d-a)}{d-c}
=\gamma\frac{-cl(c)+a-c}{d-c}+b
$$
and
\begin{align*}
\av{\varphi^2}{[c,d]}&=\frac{(2\gamma^2+2\gamma b+b^2)(a-c)
+\gamma^2(-cl^2(c))+2\gamma(b+\gamma)(-cl(c))+b^2(d-a)}{d-c}\\
&=\frac{\gamma}{d-c}\left[2(b+\gamma)(a-c)-\gamma cl^2(c)-2(b+\gamma)cl(c)\right]+b^2,
\end{align*}
so
\begin{align*}
\av{\varphi^2}{[c,d]}-\av{\varphi}{[c,d]}^2&=
\frac{\gamma^2}{d-c}\left[2(a-c)-cl^2(c)-2cl(c)\right]\\
&-\frac{\gamma^2}{(d-c)^2}\left[(a-c)^2-2c(a-c)l(c)+c^2l^2(c)\right]\\
&=\frac{\gamma^2}{(d-c)^2}\left[2(a-c)(d-c)-(a-c)^2-cdl^2(c)-2c(d-a)l(c)\right]\\
&\le\gamma^2\frac{a-c}{d-c}\left[2-\frac{a-c}{d-c}\right]\le\gamma^2,
\end{align*}
since $d\ge a$ and $\log(a/c)\ge0$ if $a\ge c.$ The last inequality follows from the fact
that the vertex of the parabola $(x,x(2-x))$ is at $(1,1).$ \item $0\le a\le c<d\le 1.$ In
this case,
$$
\av{\varphi^2}{[c,d]}-\av{\varphi}{[c,d]}^2=b^2-b^2=0.
$$
\een
We have shown that $\varphi_{a,b,\gamma}\in \BMO_{|\gamma|}(I).$ Also, using Case 2
above with $c=0,$ $d=1,$ we get $\av{\varphi_{a,b,\gamma}}I=\gamma a+b$ and
$\av{\varphi^2_{a,b,\gamma}}I=2\gamma^2a+2\gamma ab+b^2.$ Finally,
$$
\av{e^{\varphi_{a,b,\gamma}}}I=\int_0^ae^b\left(\frac at\right)^\gamma dt+\int_a^1e^b\,dt=
\begin{cases}
\ds\frac{1-\gamma+a\gamma}{1-\gamma}\;e^b\quad&\hbox{if}~\gamma<1\\
\qquad\infty~&\hbox{if}~\gamma\ge1.
\end{cases}
$$
Since $\bel{B}_{\ve}(x_1,x_1^2)=B_{\ve}(x_1,x_1^2)=e^{x_1}$ for all $\ve,$ we only need to
consider the points $x\in\Oe$ with $x_2>x_1^2.$ Then we can set
$a=1-\frac1{|\gamma|}\sqrt{\gamma^2+x_1^2-x_2}$ and $b=x_1-\gamma a,$ which yields
$\av{\varphi_{a,b,\gamma}}I=x_1,$ $\av{\varphi^2_{a,b,\gamma}}I=x_2.$ Now, if we put
$\gamma=\ve\ge1,$ we get $\bel{B}^+_{\ve}(x)=\infty.$ For $\gamma=\ve\in(0,1),$ we get
$$
\bel{B}^+_{\ve}(x)\ge\av{e^{\varphi_{a,b,\gamma}}}I=
\frac{1-\sqrt{\ve^2+x_1^2-x_2}}{1-\ve}\exp\left(x_1+\sqrt{\ve^2+x_1^2-x_2}-\ve\right)=
B^+_{\ve}(x).
$$
If we set $\gamma=-\ve\in(-\infty,0),$ we obtain
$$
\bel{B}^-_{\ve}(x)\le\av{e^{\varphi_{a,b,\gamma}}}I=
\frac{1+\sqrt{\ve^2+x_1^2-x_2}}{1+\ve}\exp\left(x_1-\sqrt{\ve^2+x_1^2-x_2}+\ve\right)=
B^-_{\ve}(x).\mbox{\qedhere}
$$
\end{proof}
\begin{lemc}
\label{l2c}
For every $x\in\Oe,$
\eq[t15]{
\bel{B}^+_{\ve}(x)\le
B^+_{\ve}(x);\qquad\bel{B}^-_{\ve}(x)\ge B^-_{\ve}(x),
}
where $0<\ve<1$ for $B^+$ and $\ve>0$ for $B^-.$
\end{lemc}
\begin{proof}
To establish \eqref{t15}, we first prove that $\bel{B}^+_{\ve}(x)\le B^+_{\ve_1}(x),\;
\bel{B}^-_{\ve}(x)\ge B^-_{\ve_1}(x)\;\forall \ve_1>\ve,\;\forall x\in\Oe,$ and take the
limit as $\ve_1\to\ve.$ (Observe that $B^+_\ve$ and $B^-_\ve$ are continuous in $\ve.$) We
need the following two results; their proofs will be postponed until the end of the proof
of Lemma~\ref{l2c}. 
\begin{lemc}
\label{l3c}
The function $B^+_{\ve}$ is locally concave and the function $B^-_{\ve}$
locally convex in $\Oe,$ i.e.
\begin{align}
\notag B^+_{\ve}(\alpha_-x^-+\alpha_+x^+)\ge \alpha_-B_{\ve}(x^-)+\alpha_+B_{\ve}(x^+)
\label{t16}
\\
B^-_{\ve}(\alpha_-x^-+\alpha_+x^+)\le \alpha_-B_{\ve}(x^-)+\alpha_+B_{\ve}(x^+)
\end{align}
for any straight-line segment with the endpoints $x^{\pm}$ that lies entirely in $\Oe$ and
any pair of nonnegative numbers $\alpha_\pm$ such that $\alpha_-+\alpha_+=1.$
\end{lemc}
\begin{lemc}
\label{l4c}
Fix $\ve.$ Take any $\ve_1>\ve.$ Then for every interval $I$ and every $\varphi\in
\BMO_{\ve}(I),$ there exists a splitting $I=I_-\cup I_+$ such that the whole straight-line
segment with the endpoints
$x^{\pm}=\left(\av{\varphi}{I_{\pm}},\av{\varphi^2}{I_{\pm}}\right)$ is inside
$\Omega_{\ve_1}.$ Moreover\textup, the splitting parameter $\alpha_+=|I_+|/|I|$ can be chosen
uniformly \textup(with respect to $\varphi$ and $I$\textup) separated from $0$ and $1.$
\end{lemc}
Assuming these lemmas for the moment, take $\varphi\in \BMO_{\ve}(I).$ Take any
$\ve_1>\ve.$ Observe that $\varphi\in \BMO_{\ve}(J)$ for any subinterval $J$ of $I.$ Split
$I$ according to the rule from Lemma~\ref{l4c}. Let $I^{0,0}=I,\quad I^{1,0}=I_-,\quad I^{1,1}=I_+.$
Now split $I_-$ and $I_+$ according to the rule from Lemma~\ref{l4c} and continue this splitting
process. By $I^{n,m}$ we denote the intervals of the $n\!$-th generation, as follows:
$I^{n,2k}=I_-^{n-1,k}$ and $I^{n,2k+1}=I_+^{n-1,k},$ so the second index runs from $0$ to
$2^n-1.$ We call the quasi-dyadic lattice so obtained $D_\varphi=D_\varphi(I).$ Let
$x^{n,m}=\left(\av{\varphi}{I^{n,m}},\av{\varphi^2}{I^{n,m}}\right).$ Since Lemma~\ref{l4c} provides
for the value of $\alpha_+$ uniformly separated from 0 and 1 on every step, we have
$$
\max_{k=0,1,\dots,2^n-1}\left\{|I^{n,k}|\right\}\longrightarrow0~\text{as}~n\to\infty.
$$
With this notation, for a given $\varphi\in\BMO_\ve(J)$ let us now introduce two sequences of
step functions $\varphi_n(s)=x_1^{n,k}$ and $s_n(s)=x_2^{n,k}-(x_1^{n,k})^2$ for $s\in
I^{n,k}.$ Note that $\varphi_n-\av\varphi I$ is the partial sum of the expansion of the
function $\varphi-\av\varphi I$ with respect to the orthonormal family of the generalized
Haar functions related to $D_\varphi(I)$
$$
h_J=
\begin{cases}
\phantom{-}\left(\frac{|J_+|}{|J|\,|J_-|}\right)^{1/2}&\text{on}~J_-,\\
&\\
-\left(\frac{|J_-|}{|J|\,|J_+|}\right)^{1/2}&\text{on}~J_+.
\end{cases}
$$
It is clear that under the assumption that the lengths of intervals $I^{n,k}$ go to zero as
$n\to\infty,$ the family $\{h_J\}_{J\in D_\varphi}$ forms a basis in $L_0^2(I)=\{\psi\in
L^2(I)\colon\av{\psi}I=0\}.$  So $\varphi_n\to\varphi$ in the $L^2\!$-norm and since
\begin{align*}
\|\varphi-\varphi_n\|^2_{L^2}&=\int_I|\varphi(s)-\varphi_n(s)|^2ds
=\sum_{I^{n,k}}\int_{I^{n,k}}|\varphi(s)-\varphi_n(s)|^2ds
\\
&=\sum_{I^{n,k}}|I^{n,k}|\bigl(x_2^{n,k}-(x_1^{n,k})^2\bigr)= \int_I s_n(s)\,ds,
\end{align*}
we can choose a subsequence $n_j$ such that $\varphi_{n_j}(s)\to\varphi(s)$ and
$s_{n_j}(s)\to0$ almost everywhere on $I.$

Now, using the statement about $B^+$ from Lemma~\ref{l3c} repeatedly, we get
\begin{eqnarray}
\nonumber B^+_{\ve_1}(x^{0,0})&\ge&\ds
\frac{|I^{1,0}|}{|I^{0,0}|}B_{\ve_1}(x^{1,0})+\frac{|I^{1,1}|}{|I^{0,0}|}B^+_{\ve_1}(x^{1,1})\\
\nonumber &&\\
\nonumber &\ge&\ds
\frac{|I^{1,0}|}{|I^{0,0}|}\frac{|I^{2,0}|}{|I^{1,0}|}B^+_{\ve_1}(x^{2,0})+
\frac{|I^{1,0}|}{|I^{0,0}|}\frac{|I^{2,1}|}{|I^{1,0}|}B^+_{\ve_1}(x^{2,1})\\
\nonumber &&\\
\label{t17}&&+\ds \frac{|I^{1,1}|}{|I^{0,0}|}\frac{|I^{2,2}|}{|I^{1,1}|}B^+_{\ve_1}(x^{2,2})+
\frac{|I^{1,1}|}{|I^{0,0}|}\frac{|I^{2,3}|}{|I^{1,1}|}B^+_{\ve_1}(x^{2,3})\\
\nonumber &&\\
\nonumber &=&\ds
\frac{|I^{2,0}|}{|I^{0,0}|}B^+_{\ve_1}(x^{2,0})+\frac{|I^{2,1}|}{|I^{0,0}|}B^+_{\ve_1}(x^{2,1})+
\frac{|I^{2,2}|}{|I^{0,0}|}B^+_{\ve_1}(x^{2,2})+\frac{|I^{2,3}|}{|I^{0,0}|}B^+_{\ve_1}(x^{2,3})\\
\nonumber &&\\
\nonumber &\ge&\ds \frac1{|I^{0,0}|}\sum_{k=0}^{2^n-1}|I^{n,k}|B^+_{\ve_1}(x^{n,k})
=\frac1{|I|}\int_I e^{\varphi_n(s)}b_+(s_n(s))\,ds,
\end{eqnarray}
where
$$
b_+(t)=\frac{1-\sqrt{\ve_1^2-t}}{1-\ve_1}\exp\left(\sqrt{\ve_1^2-t}-\ve_1\right).
$$
The last equality is just the statement $B^+_{\ve_1}(x^{n,k})=e^{\varphi_n(s)}b_+(s_n(s)),$
for $s\in I^{n,k}.$

Likewise, applying the corresponding statement from Lemma~\ref{l3c} repeatedly, we obtain
\eq[t17-]{
B^-_{\ve_1}(x^{0,0})\le
\frac1{|I^{0,0}|}\sum_{k=0}^{2^n-1}|I^{n,k}|B^-_{\ve_1}(x^{n,k})=\frac1{|I|}\int_I
e^{\varphi_n(s)}b_-(s_n(s))\,ds.
}
Here
$$
b_-(t)=\frac{1+\sqrt{\ve_1^2-t}}{1+\ve_1}\exp\left(-\sqrt{\ve_1^2-t}+\ve_1\right).
$$
For functions $\varphi$ bounded from above we can pass to the limit in \eqref{t17} and
\eqref{t17-} using the dominated convergence theorem. Therefore, for such functions
$\varphi\in\BMO_\ve(J)$ we have the double inequality 
\eq[estim]{
B^-_{\ve_1}(\av{\varphi}I,\av{\varphi^2}I)\le \frac1{|I|}\int_I e^{\varphi(s)}\,ds\le
B^+_{\ve_1}(\av{\varphi}I,\av{\varphi^2}I).
}
It remains to approximate an arbitrary function $\psi\in \BMO_\ve(I)$ by its cut-offs in a
standard manner; namely, we take
$$
\psi_m(s)=
\begin{cases}
\psi(s)\quad&\text{if }\psi(s)\le m
\\
m&\text{if }\psi(s)>m.
\end{cases}
$$
If we denote $J_1=\{s\in J\colon \psi(s)\le m\}$ and  $J_2=\{s\in J\colon \psi(s)>\},$ we
have the following identity
\begin{align*}
\bigl(\av{\psi^2}J&-(\av\psi J)^2\bigr)-\bigl(\av{\psi_m^2}J-(\av{\psi_m}J)^2\bigr)
\\
=&\frac{|J_2|}{|J|}\bigl(\av{\psi^2}{J_2}-(\av{\psi}{J_2})^2\bigr)+
\frac{|J_2|\,|J_1|}{|J|^2}\bigl(\av{\psi}{J_2}-m)\bigr)
\bigl(\av{\psi}{J_2}+m-2\av{\psi}{J_1}\bigr)\ge0,
\end{align*}
which implies that $\psi_m$ is in $\BMO_\ve(I)$ if $\psi$ is. Therefore, for $\varphi=\psi_m$
inequalities~\eqref{estim} hold and we can pass to the limit as $m\to\infty.$ Clearly, the
averages of $\psi_m$ converge to the averages of $\psi$ and the values of
$B^\pm_{\ve_1}(\av{\psi_m}{},\av{\psi_m^2}{})$ converge to
$B^\pm_{\ve_1}(\av{\psi}{},\av{\psi^2}{})$ because of continuity of the functions $B^\pm.$
Due to the monotone convergence of $\psi_m$ we can pass to the limit under the integral.
Taking first the supremum and then infimum over all $\psi\in \BMO_\ve(I)$ with
$\av{\psi}I=x_1$ and $\av{\psi^2}I=x_2,$ we obtain the inequalities
$$
B^+_{\ve_1}(x)\ge \bel{B}^+_{\ve}(x),\qquad B^-_{\ve_1}(x)\le \bel{B}^-_{\ve}(x),
$$
thus proving the lemma.
\end{proof}
\begin{proof}[Proof of Lemma~\ref{l3c}] To prove the lemma, we need to check that
\eq[t20.1]{
\mp\frac{\partial^2 B_\ve^\pm}{\partial x_i\partial x_j}
}
is a nonnegative matrix. Direct calculation yields
\begin{align*}
\frac{\partial B_\ve^\pm}{\partial x_1}&=
\frac{1-x_1\mp\sqrt{\ve^2+x_1^2-x_2}}{1\mp\ve}
\exp\left\{x_1\pm\sqrt{\ve^2+x_1^2-x_2}\mp\ve\right\},
\\
\frac{\partial B_\ve^\pm}{\partial x_2}&=
\frac1{2(1\mp\ve)}
\exp\left\{x_1\pm\sqrt{\ve^2+x_1^2-x_2}\mp\ve\right\},
\\
\frac{\partial^2 B_\ve^\pm}{\partial x_1^2}&=
\mp\frac{\left(x_1\pm\sqrt{\ve^2+x_1^2-x_2}\right)^2}
{\sqrt{\ve^2+x_1^2-x_2}(1\mp\ve)}
\exp\left\{x_1\pm\sqrt{\ve^2+x_1^2-x_2}\mp\ve\right\},
\\
\frac{\partial^2 B_\ve^\pm}{\partial x_1\partial x_2}&=
\pm\frac{x_1\pm\sqrt{\ve^2+x_1^2-x_2}}
{2\sqrt{\ve^2+x_1^2-x_2}(1\mp\ve)}
\exp\left\{x_1\pm\sqrt{\ve^2+x_1^2-x_2}\mp\ve\right\},
\\
\frac{\partial^2 B_\ve^\pm}{\partial x_2^2}&=
\mp\frac1{4\sqrt{\ve^2+x_1^2-x_2}(1\mp\ve)}
\exp\left\{x_1\pm\sqrt{\ve^2+x_1^2-x_2}\mp\ve\right\}.
\end{align*}
Therefore, the quadratic form of the matrix \eqref{t20.1} is
\begin{align}
\notag\mp&\sum_{i,j=1}^2
\frac{\partial^2 B_\ve^\pm}{\partial x_i\partial x_j}\Delta_i\Delta_j=&\\
\label{t20.2}&\frac{\left(\left(x_1\pm\sqrt{\ve^2+x_1^2-x_2}\right)\Delta_1
-\frac12\Delta_2\right)^2}{\sqrt{\ve^2+x_1^2-x_2}(1\mp\ve)}
\exp\left\{x_1\pm\sqrt{\ve^2+x_1^2-x_2}\mp\ve\right\}
\ge0,
\end{align}
which establishes the result.
\end{proof}
\begin{proof}[Proof of Lemma~\ref{l4c}]
We fix an interval $I$ and a function $\varphi\in \BMO_{\ve}(I).$ We
now explicitly construct an algorithm to find the splitting $I=I_-\cup I_+,$ i.e. choose the
splitting parameters $\alpha_\pm=|I_\pm|/|I|.$ As before,
$x_1^{\pm}=\av{\varphi}{I_\pm},$ $x_2^{\pm}=\av{\varphi^2}{I_\pm}.$ Also,
put $x^0_1=\av{\varphi}I$ and $x^0_2=\av{\varphi^2}I.$ Lastly, by $[s,t]$
we will denote the straight-line segment connecting two points $s$ and $t$ in the plane.

First, we take $\alpha_-=\alpha_+=\frac12$ (see Fig. \ref{f11}). If the whole segment $[x^-,x^+]$ is in
$\Omega_{\ve_1},$ we fix this splitting. Assuming it is not the case, there exists a point
$x$ on this segment with $x_2-x_1^2>\ve_1^2.$ Observe that only one of the segments
$[x^-,x^0]$ and $[x^+,x^0]$ contains such points. Call the corresponding endpoint ($x^-$ or
$x^+$) $\xi.$
\begin{figure}[ht]
\begin{center}
\begin{picture}(200,200)
\thinlines
\put(100,0){\vector(0,1){180}}
\put(10,20){\vector(1,0){180}}
\linethickness{.5pt}
\thicklines
\qbezier[1000](20,82)(100,-42)(180,82)
\qbezier[1000](20,162)(100,15)(180,162)
\qbezier[1000](20,168)(100,15)(180,168)
\qbezier[1000](28,77)(28,77)(138,101)
\put(30,69){\footnotesize $x^-$}
\put(83,80){\footnotesize $x^0$}
\put(138,93){\footnotesize $\xi$}
\put(28,77){\circle*{2}}
\put(83,89){\circle*{2}}
\put(138,101){\circle*{2}}
\put(180,80){\footnotesize $x_2=x_1^2$}
\put(180,160){\footnotesize $x_2=x_1^2+\ve^2$}
\put(180,175){\footnotesize $x_2=x_1^2+\ve_1^2$}
\end{picture}
\caption{The initial splitting: $\alpha_-=\alpha_+=\frac12,~\xi=x^+.$}
\label{f11}
\end{center}
\end{figure}
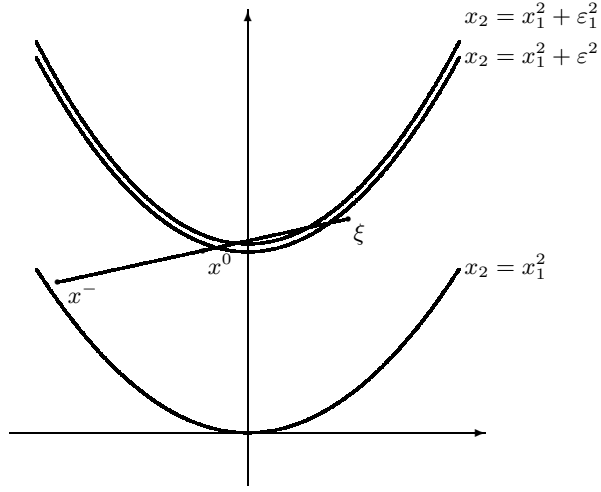
Its position is completely defined by the choice of $\alpha_+.$ Define the function $\rho$ as
follows: $\rho(\alpha_+)=\max_{x\in [\xi,x^0]}\{x_2-x_1^2\}.$ By assumption,
$\rho\left(\frac12\right)>\ve_1^2.$ We will now change $\alpha_+$ so that $\xi$ approaches
$x^0,$ i.e. we will increase $\alpha_+$ if $\xi=x^+$ and decrease it if $\xi=x^-.$ We stop
when $\rho(\alpha_+)=\ve_1^2$ and fix that splitting. It remains to check that such a moment
occurs at all and that the corresponding $\alpha_+$ is separated from 0 and 1. Without loss
of generality, assume that $\xi=x^+.$ Let $I=[a,b].$ Since $\varphi\in L^2(I),$ the functions
$\xi_1(\alpha_+)=\frac1{\alpha_+}\int_{b-|I|\alpha_+}^b \varphi(w)\,dw$ and
$\xi_2(\alpha_+)=\frac1{\alpha_+}\int_{b-|I|\alpha_+}^b \varphi^2(w)\,dw$ are continuous on
the interval $(0,1]$ and $\xi(1)=x^0.$ Therefore, $\rho$ is continuous on $(0,1].$ Since
$\rho\left(\frac12\right)>\ve_1^2$ and $\rho(1)\le \ve^2<\ve_1^2$ (recall, $x^0\in \Oe$), we
conclude that there is a point $\alpha_+\in\left[\frac12,1\right]$ with
$\rho(\alpha_+)=\ve_1^2$ (Fig. \ref{f12}).

Having just proved that the desired point exists, we need to check that the corresponding
$\alpha_+$ is not too close to 0 or 1. If $\xi=x^+,$ we have $\alpha_+>\frac12$ and
$\xi_1-x_1^0=x_1^+-x_1^0=\alpha_-(x_1^+-x_1^-).$ Analogously, if $\xi=x^-,$ we have
$\alpha_->\frac12$ and $\xi_1-x_1^0=x_1^--x_1^0=\alpha_+(x_1^--x_1^+).$ Thus
$|\xi_1-x_1^0|=\min\{\alpha_\pm\}|x_1^--x_1^+|.$
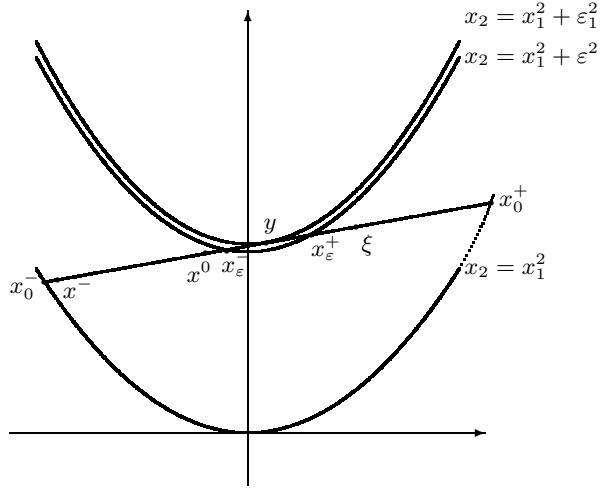
\begin{figure}[ht]
\begin{center}
\begin{picture}(200,200)
\thinlines
\put(100,0){\vector(0,1){180}}
\put(10,20){\vector(1,0){180}}
\linethickness{.5pt}
\thicklines
\qbezier[1000](20,82)(100,-42)(180,82)
\qbezier[1000](20,162)(100,15)(180,162)
\qbezier[1000](20,168)(100,15)(180,168)
\qbezier[1000](23,77)(23,77)(192,107)
\qbezier[20](180,82)(190,100)(193,110)
\put(23,77){\circle*{2}}
\put(28,78){\circle*{2}}
\put(141,98){\circle*{2}}
\put(84,88){\circle*{2}}
\put(92,89){\circle*{2}}
\put(107,92){\circle*{2}}
\put(124,95){\circle*{2}}
\put(192,107){\circle*{2}}
\put(180,80){\footnotesize $x_2=x_1^2$}
\put(180,160){\footnotesize $x_2=x_1^2+\ve^2$}
\put(180,175){\footnotesize $x_2=x_1^2+\ve_1^2$}
\put(28,71){\footnotesize $x^-$}
\put(75,77){\footnotesize $x^0$}
\put(88,81){\footnotesize $x_\ve^-$}
\put(104,97){\footnotesize $y$}
\put(122,87){\footnotesize $x_\ve^+$}
\put(141,89){\footnotesize $\xi$}
\put(193,106){\footnotesize $x_0^+$}
\put(8,73){\footnotesize $x_0^-$}
\end{picture}
\caption{The stopping time: $[x^-,\xi]$ is tangent to the parabola $x_2=x_1^2+\ve^2.$}
\label{f12}
\end{center}
\end{figure}
For the stopping value of $\alpha_+,$ the straight line through the points $x^-,x^+$ and
$x^0$ is tangent to the parabola $x_2=x_1^2+\ve_1^2$ at some point $y.$ The equation of this
line is, therefore, $x_2=2x_1y_1-y_1^2+\ve_1^2.$ The line intersects the graph of
$x_2=x_1^2+\ve^2$ at the points
$$
x_\ve^{\pm}=\left(y_1\pm\sqrt{\ve_1^2-\ve^2},y_2\pm 2y_1\sqrt{\ve_1^2-\ve^2}\right)
$$
and the graph of $x_2=x_1^2$ at the points
$$
x_0^{\pm}=(y_1\pm\ve_1,y_2\pm2y_1\ve_1).
$$
We then have
$$
[x_\ve^-,x_\ve^+]\subset[x^0,\xi]\subset[x^-,x^+]\subset[x_0^-,x_0^+]
$$
and, therefore,
$$
\begin{array}{lll}
2\sqrt{\ve_1^2-\ve^2}&=&|(x_\ve^+)_1-(x_\ve^-)_1|\le|x_1^0-\xi_1|=\min\{\alpha_\pm\}|x_1^+-x_1^-|\\
&&\\
&\le&\min\{\alpha_\pm\}|(x_0^+)_1-(x_0^-)_1|=\min\{\alpha_\pm\}2\ve_1,
\end{array}
$$
which implies
$$
\sqrt{1-\left(\frac{\ve}{\ve_1}\right)^2}\le \alpha_+\le 1-\sqrt{1-\left(\frac{\ve}{\ve_1}\right)^2}.
$$
As promised, this estimate does not depend on $\varphi$ or $I.$
\end{proof}
\subsection{How to find the Bellman function}
\label{howtobelc}
We first observe that the Bellman functions $\bel{B}^\pm$ must be of the form
\eq[t18]{
\bel{B}^\pm_{\ve}(x)=\exp\left\{x_1+w^\pm_{\ve}(x_2-x_1^2)\right\}
}
for some positive functions
$w^\pm$ on $[0,\ve^2]$ such that $w^\pm_{\ve}(0)=0.$

Indeed, fix an interval $I.$ Then $\varphi\in \BMO_{\ve}(I)$ if and only if $\varphi+c\in
\BMO_{\ve}(I),$ where $c$ is an arbitrary constant. Let $\tilde{\varphi}=\varphi+c.$ We have
(all averages are over $I$) $\ave{\tilde{\varphi}}=\ave{\varphi}+c,$ $\ave{\tilde{\varphi}}^2=\ave{\varphi^2}+2c\ave{\varphi}+c^2,$ and
$\ave{e^{\tilde{\varphi}}}=e^c\ave{e^\varphi}.$ Then
$$
\sup_{\varphi\in
\BMO_{\ve}(I)}\!\!\left\{\ave{e^{\tilde{\varphi}}}:\,\ave{\varphi}=x_1,\ave{\varphi}^2=x_2\right\}=
e^c\!\!\!\!\sup_{\varphi\in
\BMO_{\ve}(I)}\!\!\left\{\ave{e^\varphi}:\,\ave{\varphi}=x_1,\ave{\varphi^2}=x_2\right\}
$$
or
\begin{align*}
\sup_{\tilde{\varphi}\in
\BMO_{\ve}(I)}\!\!\left\{\ave{e^{\tilde{\varphi}}}:\,\ave{\tilde{\varphi}}=x_1+c,
\ave{\tilde{\varphi}^2}=x_2+2cx_1+c^2\right\}&\\= e^c\!\!\!\!\sup_{\varphi\in
\BMO_{\ve}(I)}\!\!&\left\{\ave{e^\varphi}:\,\ave{\varphi}=x_1,\ave{\varphi^2}=x_2\right\}.
\end{align*}
Completely analogous statements with $\inf$ instead of $\sup$ can be made. Altogether, we get
or
$$
\bel{B}^\pm_{\ve}(x_1+c,x_2+2cx_1+c^2)=e^c\bel{B}^\pm_{\ve}(x_1,x_2).
$$
Setting $c=-x_1,$ and omitting the index $\ve$ we get
$$
\bel{B}^\pm(0,x_2-x_1^2)=e^{-x_1}\bel{B}^\pm(x_1,x_2).
$$
By Jensen's inequality ($\ave{e^\varphi}\ge e^{\ave{\varphi}}$), we obtain
$\bel{B}^\pm(0,x_2-x_1^2)\ge 1.$ Hence, there exists a positive function $w^\pm=\log
\bel{B}^\pm(0,\cdot)$ defined on the interval $[0,\ve^2]$ such that \eqref{t18} holds.
Furthermore, $x_2=x_1=0$ if and only if $\varphi=0.$ Thus $\bel{B}^\pm(0,0)=1$ and
$w^\pm(0)=0.$

The successful Bellman function candidate $B$ (we will omit the index $\pm$ when no confusion results) must be of the form \eqref{t18}. Moreover, to use the machinery of Lemma~\ref{l2c}, we need the statements of Lemma~\ref{l3c} to hold. So we want
\eq[t19]{
\mp\frac{\partial^2 B^\pm}{\partial x_i\partial x_j}
}
to be a nonnegative matrix.

Using \eqref{t18}, we get
\begin{align*}
\frac{\partial B}{\partial x_1}&=(1-2x_1w')B,\smallskip \\
\frac{\partial B}{\partial x_2}&=w'B,\smallskip \\
\frac{\partial^2 B}{\partial x_1^2}&=\left((1-2x_1w')^2-2w'+4x_1^2w''\right)B,\smallskip \\
\frac{\partial^2 B}{\partial x_1\partial x_2}&=\left(w'(1-2x_1w')-2x_1w''\right)B,\smallskip \\
\frac{\partial^2 B}{\partial x_2^2}&=\left((w')^2+w''\right)B.
\end{align*}
Matrix \eqref{t19} turns into
\eq[t20]{
\mp\left[
\begin{array}{cc}
\ds \frac{\partial^2 B^\pm}{\partial x_1^2}&\ds \frac{\partial^2 B^\pm}{\partial x_1\partial x_2}\smallskip \\
\ds \frac{\partial^2 B^\pm}{\partial x_1\partial x_2}&\ds \frac{\partial^2 B^\pm}{\partial x_2^2}
\end{array}
\right]=
\mp B^\pm
\left[
\begin{array}{cc}
1&-2x_1\\
0&1
\end{array}
\right]
R
\left[
\begin{array}{cc}
1&0\\
-2x_1&1
\end{array}
\right],
}
where
\eq[t21]{
R=\left[
\begin{array}{cc}
1-2w'&w'\\
w'&(w')^2+w''
\end{array}
\right]. 
} 
For the extremal function (if any) we must have equality at every step in
\eqref{t17} and \eqref{t17-} in Lemma~\ref{l2c}, so the matrix \eqref{t19} has to be degenerate.
Because of the representation \eqref{t20} and \eqref{t21}, this translates into 
\eq[t22]{
(1-2w')\left((w')^2+w''\right)=(w')^2, 
} 
and the non-negativity condition~\eqref{t19} is
equivalent to the inequality 
\eq[t23]{ 
\pm(2(w^\pm)'-1)\ge0. 
} We solve equation \eqref{t22}
\begin{align*}
\ds(1-2w')w''&=2(w')^3\\
\ds\left(\frac1{2(w')^3}-\frac1{(w')^2}\right)w''&=1\\
\ds\left(\frac1{w'}-\frac1{4(w')^2}\right)'&=1\\
\ds\frac1{w'}-\frac1{4(w')^2}&=t+const\\
\ds-\left(1-\frac1{2w'}\right)^2&=t+const.
\end{align*}
This implies that the constant has to be non-positive. We parametrize the family of possible
solutions by a positive parameter $\delta$ setting $const=-\delta^2.$ Then
$$
\left(1-\frac1{2w'}\right)^2=\delta^2-t
$$
and 
\eq[t24]{ 
1-\frac1{2w'}=\pm\sqrt{\delta^2-t}. 
} 
We see that the solution is defined on
the interval $[0,\delta^2].$ Condition \eqref{t23} with ``$+$'' means that $w'\ge\half$. This
requires the ``$+$'' sign in~\eqref{t24} and this square root has to be strictly less than
$1.$ Therefore, the only feasible solution for $w^+$ is that for $\delta<1.$ We get the
solution for $w^-$ by choosing the ``$-$'' sign in~\eqref{t24}. It works for all $\delta>0.$
Thus, equation \eqref{t24} gives
$$
(w^\pm)'=\frac1{2(1\mp\sqrt{\delta^2-t})}
$$
and, taking into account that $w(0)=0,$ we obtain
$$
w^\pm(t)=\frac12\int_0^t \frac1{1\mp\sqrt{\delta^2-s}}\,ds=\log\frac{1\mp\sqrt{\delta^2-t}}{1\mp\delta}\pm\sqrt{\delta^2-t}\mp\delta,
$$
which, together with \eqref{t18}, gives \eqref{t12}
$$
B^\pm_{\delta}(x)=\frac{1\mp\sqrt{\delta^2+x_1^2-x_2}}{1\mp\delta}\exp\left(x_1\pm\sqrt{\delta^2+x_1^2-x_2}\mp\delta\right).
$$
\subsection{How to find the extremal function}
\label{howtoextc} 
We now show how to find the extremal function that appeared without an
explanation in the proof of Lemma~\ref{l1c}. As mentioned in the previous section, for the extremal
function there is equality at every step in the chain of inequalities \eqref{t17}. Thus in
the splitting process we only proceed along the vector field defined by the kernel vectors of
the matrix \eqref{t19}. The quadratic form of that matrix is given by~\eqref{t20.2}:
\eq[t25]{
\begin{array}{ll}
\ds \mp\sum_{i,j=1}^2\frac{\partial^2 B^\pm_\delta}{\partial x_i \partial x_j}\Delta_i\Delta_j=\\
\ds \frac{\left(\left(x_1\pm\sqrt{\delta^2+x_1^2-x_2}\right)\Delta_1-\frac12\Delta_2\right)^2}{\sqrt{\delta^2+x_1^2-x_2}\,(1\mp\delta)}
\exp\left\{x_1\pm\sqrt{\delta^2+x_1^2-x_2}\mp\delta\right\}.
\end{array}
} 
Hence, the trajectories along which $B$ is a linear function are given by 
\eq[t26]{
\left(x_1\pm\sqrt{\delta^2+x_1^2-x_2}\right)dx_1=\frac12dx_2. 
}
Introducing the variable
$t=\pm\sqrt{\delta^2+x_1^2-x_2},$ we have $t^2=\delta^2+x_1^2-x_2$ and
$2t\,dt=2x_1\,dx_1-dx_2.$ Replacing $\frac12dx_2$ in \eqref{t26} by $x_1\,dx_1-t\,dt,$ we get
$t\,dx_1=-t\,dt,$ i.e. $t=c-x_1$ and 
\eq[t26.1]{
x_2=\delta^2+x_1^2-t^2=2cx_1+\delta^2-c^2. 
}
The corresponding trajectories are straight lines tangent to the upper boundary
$x_2=x_1^2+\delta^2$ of $\Od$ at the point $x=(c,c^2+\delta^2).$ Consider the following two
families of such straight-line segments
$$
\begin{array}{c}
\ds\omega^+_\delta(c)=\left\{x=(x_1,2cx_1+\delta^2-c^2):c-\delta\le x_1\le c\right\};\\
\ds\omega^-_\delta(c)=\left\{x=(x_1,2cx_1+\delta^2-c^2):c\le x_1\le c+\delta\right\}.
\end{array}
$$
Each of these families covers the whole domain, i.e.
$$
\Od=\bigcup_{c\in\mathbb{R}}\omega^+_\delta(c)=\bigcup_{c\in\mathbb{R}}\omega^-_\delta(c).
$$
Furthermore, $B^+$ is a linear function on each segment $\omega^+_\delta(c),$ while $B^-$ is
a linear function on each segment $\omega^-_\delta(c).$ Indeed, since
$\sqrt{\delta^2+x_1^2-x_2}=|x_1-c|$ on the line $x_2=2cx_1+\delta^2-c^2,$ we have
\begin{align*}
B^+_{\delta}(x_1,2cx_1+\delta^2-c^2)=\frac{1+x_1-c}{1-\delta}e^{c-\delta}~~\text{for}~~c-\delta\le x_1\le c;\\
B^-_{\delta}(x_1,2cx_1+\delta^2-c^2)=\frac{1+x_1-c}{1+\delta}e^{c+\delta}~~\text{for}~~c\le x_1\le c+\delta.
\end{align*}
Therefore, if both points $x^\pm$ are on a segment $\omega^+_\delta(c)$ or
$\omega^-_\delta(c),$ we have equality in the corresponding line in \eqref{t16} (with
$\delta=\ve$).\footnote{To avoid misunderstanding, we note that $\pm$ in $x^\pm$ and in
$\omega_\delta^\pm$ are independent: $x^\pm$ are two points in the domain $\Od$ whose convex
combination is the point $x,$ while $\pm$ in $\omega_\delta^\pm$ means that we consider
either $B^+$ or $B^-,$ as appropriate.}

Note that we have one more ``acceptable trajectory,'' the envelope of the segments
$\omega^+_\delta(c)$ (or $\omega^-_\delta(c)$) the parabola $x_2=x_1^2+\delta^2.$

Let $x^0$ be an arbitrary point inside $\Od.$ Then we make the splitting so that $x^-$ is on
the boundary $x_2=x_1^2+\delta^2$ and the segment $\omega^+_\delta(x_1^-)$ passes through the
point $x^0.$ Every point on that segment satisfies the equation
$$
x_2=2x_1^-x_1+\delta^2-(x_1^-)^2,
$$
so $x_1^-=x_1^0+\sqrt{\delta^2+(x_1^0)^2-x_2^0}.$ We choose the second endpoint $x^+$ to be
the point of intersection of $\omega^+_\delta(x_1^-)$ and the lower boundary of $\Od,$
$x_2=x_1^2.$ This is equivalent to letting $\varphi$ be constant on $I_+.$ Then
$x_2^+=(x_1^+)^2=2x_1^-x_1^++\delta^2-(x_1^-)^2$ and, hence, $x_1^+=x_1^--\delta.$

Assume that $\varphi_c$ is the extremal function (defined on $[0,1]$) that corresponds to the
point $(c,c^2+\delta^2)$ on the upper boundary. Then for $\varphi|_{I_-}$ we have to take the
function $\varphi_{x_1^-}$ rescaled to the interval $I_-.$ So, if $I=[0,1],$ then
$I_-=[0,\alpha_-],$ $I_+=[\alpha_-,1],$ and 
\eq[extrf]{ 
\varphi(t)=
\begin{cases}
\varphi_{x_1^-}\bigl(\frac{t}{\alpha_-}\bigr),\quad & 0\le t<\alpha_-
\\
\quad x_1^+,&\alpha_-\le t\le1.
\end{cases}
} 
We have defined the extremal function $\varphi$ for an arbitrary point of $\Od$ under the
assumption that the extremal functions $\varphi_c$ for the upper boundary are known. Note that
it is sufficient to find one of these functions, say $\varphi_0$, because
$\varphi_c=\varphi_0+c$. Indeed, it is clear that $\varphi_0$ and $\varphi_0+c$ have the same
\BMO-norms and
$$
\av{\varphi_0+c}{}=c,\qquad\av{(\varphi_0+c)^2}{}=\av{\varphi_0^2}{}+2c\av{\varphi_0}{}+c^2=\delta^2+c^2.
$$
Let the point $x^0$ approach the point $x^-$ along the upper boundary, i.~e. let
$\alpha_+\to0.$ If we assume that the extremal function smoothly depends on the point $x^0,$
then the function $\varphi$ in~\eqref{extrf} coincides up to terms of the first order in
$\alpha_+$ with the function $\varphi_{x_1^0}:$
\begin{align*}
x_1^0&=\alpha_-x_1^-+\alpha_+x_1^+=(1-\alpha_+)x_1^-+\alpha_+(x_1^--\delta)=x_1^--\alpha_+\delta,
\\
x_2^0&=2x_1^-x_1^0+\delta^2-(x_1^-)^2=(x_1^0)^2-(x_1^--x_1^0)^2+\delta^2=(x_1^0)^2+(1-\alpha_+^2)\delta^2
\approx(x_1^0)^2+\delta^2.
\end{align*}
Therefore
$$
\varphi_{x_1^-}\Bigl(\frac{t}{\alpha_-}\Bigr)\approx\varphi_{x_1^0}(t)
$$
up to terms of the first order in $\alpha_+$. Since
$$
\varphi_{x_1^-}\Bigl(\frac{t}{\alpha_-}\Bigr)=\varphi_0\Bigl(\frac{t}{\alpha_-}\Bigr)+x_1^-=
\varphi_0\Bigl(\frac{t}{1-\alpha_+}\Bigr)+x_1^-\approx
x_1^-+\varphi_0(t)+\alpha_+t\varphi_0'(t)
$$
and
$$
\varphi_{x_1^0}(t)=\varphi_0(t)+x_1^0 =x_1^-+\varphi_0(t)-\alpha_+\delta,
$$
we have
\begin{align*}
t\varphi_0'(t)&=-\delta,
\\
\varphi_0(t)&=-\delta\log t+{\rm const}.
\end{align*}
Condition $\av{\varphi_0}{}=0$ implies 
$$
\varphi_0(t)=\delta\left(\log\frac1t-1\right). 
$$ 
This yields the function we used to prove Lemma~\ref{l1c}.
\section{The dyadic case}
To prove Theorem \ref{t2d}, we follow the procedure of the continuous case. Namely, we first produce
extremal functions $\varphi_\pm\in \BMO^d_\ve(I)$ with appropriate averages, for which
$\av{e^{\varphi_\pm}}I=B^\pm_{\delta^\pm(\ve)}.$ This proves that $\bel{B}^{d+}_\ve\ge
B^+_{\delta^+(\ve)}$ and $\bel{B}^{d-}_\ve\le
B^-_{\delta^-(\ve)}.$ Then, we use a concavity-type result similar to Lemma~\ref{l3c}, which allows us
to run the inductive machine of Lemma~\ref{l2c} to prove that the converse inequalities.

\begin{lemd}
\label{l1d}
For every $x\in\Oe,$
\eq[1d]{
\bel{B}^{d+}_{\ve}(x)\ge
B^+_{\delta^+(\ve)}(x),~~~ \bel{B}^{d-}_{\ve}(x)\le B^-_{\delta^-(\ve)}(x).
}
\end{lemd}
\begin{proof}
Let $I=[0,1].$ We prove \eqref{1d} by explicitly finding functions $\varphi_+,\varphi_-\in
\BMO^d_\ve(I)$ for every $x\in\Omega_\ve$ such that
$(\av{\varphi_\pm}I,\av{\varphi_\pm^2}I)=(x_1,x_2)$ and
$$
\av{e^{\varphi_+}}I=B^+_{\delta^+(\ve)}(x),~~~
\av{e^{\varphi_-}}I=B^-_{\delta^-(\ve)}(x).
$$
As before, we only need to consider $\ve>0.$

Fix $\ve>0.$ Let the function $\varphi_0$ be defined on $I=(0,1]$ as follows:
$$
\left.\varphi_0\right|_{\left(2^{-(k+1)},2^{-k}\right]}=(k-1)a,~~k=0,1,...,
$$
with the constant $a$ to be determined later (see Fig. \ref{fi}).
\begin{figure}[ht]
\begin{center}
\begin{picture}(200,150)
\thinlines
\put(0,0){\vector(0,1){150}}
\put(-5,40){\vector(1,0){215}}
\put(200,38){\line(0,1){4}}
\put(100,38){\line(0,1){4}}
\put(50,38){\line(0,1){4}}
\put(25,38){\line(0,1){4}}
\put(12,38){\line(0,1){4}}
\put(6,38){\line(0,1){4}}
\put(3,38){\line(0,1){4}}
\put(195,32){\tiny $1$}
\put(95,32){\tiny $\frac12$}
\put(45,32){\tiny $\frac14$}
\put(20,32){\tiny $\frac18$}
\put(6,32){\tiny $\frac1{16}$}
\put(-2,20){\line(1,0){4}}
\put(-2,40){\line(1,0){4}}
\put(-2,60){\line(1,0){4}}
\put(-2,80){\line(1,0){4}}
\put(-2,100){\line(1,0){4}}
\put(-2,120){\line(1,0){4}}
\put(-2,140){\line(1,0){4}}
\put(-14,22){\tiny $-a$}
\put(-8,62){\tiny $a$}
\put(-12,82){\tiny $2a$}
\put(-12,102){\tiny $3a$}
\put(-12,122){\tiny $4a$}
\put(-12,142){\tiny $5a$}
\thicklines
\put(100,20){\line(1,0){100}}
\put(50,40){\line(1,0){50}}
\put(25,60){\line(1,0){25}}
\put(12,80){\line(1,0){13}}
\put(6,100){\line(1,0){6}}
\put(3,120){\line(1,0){3}}
\put(1,140){\line(1,0){2}}
\end{picture}
\caption{The function $\varphi_0.$ } 
\label{fi}
\end{center}
\end{figure}
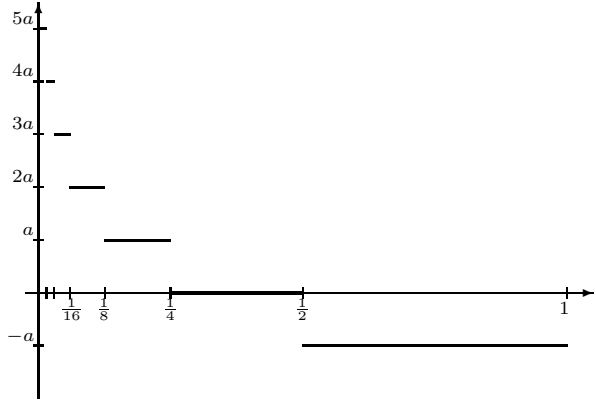
We now calculate the $\BMO^d$ norm of $\varphi_0$ and choose $a$ so that
$\|\varphi_0\|_{\BMO^d}=\ve.$ The only dyadic intervals on which $\varphi_0$ is not constant
and, hence, $\av{\varphi_0^2}I-\av{\varphi_0}I^2\ne 0$ are the ones with $0$ as
their left endpoint. Let $I_n=\left(0,2^{-n}\right].$ Then
$$
\av{\varphi_0}{I_n}=2^n\int_0^{1/2^n}\varphi_0(s)\,ds=2^n\sum_{k=n-1}^{\infty}\frac{ka}{2^{k+2}}=
\frac{a}42^n\left(\frac12\right)^{n-2}\!\!\!n=an
$$
and
$$
\av{\varphi_0^2}{I_n}=2^n\int_0^{1/2^n}\varphi_0^2(s)\,ds=2^n\sum_{k=n-1}^{\infty}\frac{k^2a^2}{2^{k+2}}=
\frac{a^2}4 2^n\left(\frac12\right)^{n-2}\left(n^2+2\right)=a^2(n^2+2),
$$
where we have used the identities
$$
\sum_{k=N-1}^{\infty}k\left(\frac12\right)^k=\left(\frac12\right)^{N-2}N,~~~~
\sum_{k=N-1}^{\infty}k^2\left(\frac12\right)^k=\left(\frac12\right)^{N-2}(N^2+2).
$$
Then
$$
\begin{array}{lll}
\|\varphi_0\|^2_{\BMO^d}&=&\ds \sup_{J\,\text{dyad}\,\subset I}\left\{\av{\varphi_0^2}{J}-\av{\varphi_0}{J}^2\right\}\\
&&\\
&=&\ds \sup_{n}\left\{\av{\varphi_0^2}{I_n}-\av{\varphi_0}{I_n}^2\right\}=\sup_n\left\{a^2(n^2+2)-a^2n^2\right\}=2a^2.
\end{array}
$$
Setting $\|\varphi_0\|_{\BMO^d}=\ve,$ we get $a=\ve/\sqrt2.$ Now,
$$
\av{e^{\varphi_0}}I=\sum_{k=-1}^{\infty}\frac{e^{ka}}{2^{k+2}}=\sum_{k=-1}^{\infty}\frac14\left(\frac{e^a}2\right)^k.
$$
The latter sum converges if and only if $e^a<2,$ i.e. $a<\log 2.$ In this case,
\eq[4]{
\av{e^{\varphi_0}}I=\frac{e^{-\ve/\sqrt2}}{2-e^{\ve/\sqrt2}}.
}
In terms of $\ve_0^d$ from Theorem~\ref{t1d}, we obtain the following crucial estimate
$$
\ve_0^d\le \sqrt2\log 2.
$$
Likewise,
$$
\av{e^{-\varphi_0}}I=\frac{e^{\ve/\sqrt2}}{2-e^{-\ve/\sqrt2}}
$$
for
arbitrary $\ve>0$.

We now use $\varphi_0$ to construct the desired functions $\varphi_\pm.$ Let
$$
r_1=\rtde;~~r_2=\sqrt{\delta^2-x_2+x_1^2};~~\beta=r_2-r_1;~~\gamma=r_2-\delta;~~
\alpha=\frac{\delta-r_2}{\delta-r_1}.
$$
Here $\delta$ will mean either $\delta^+$ or $\delta^-,$ depending on the context. Define
$\tilde{\varphi}_\pm$ on $I$ by 
$$
\tilde{\varphi}_\pm(t)=x_1\pm\tilde{\psi}(t),\text{~~where~~}\tilde{\psi}(t)=
\begin{cases}
\varphi_0\left(\frac{t}{\alpha}\right)+\beta & \mbox{for}~0<t<\alpha\\
\gamma & \mbox{for}~\alpha<t<1.
\end{cases} 
$$
Observe that $\av{\varphi_0}I=0,~\av{\varphi_0^2}I=\ve^2.$ Since
$\alpha=\gamma/(\gamma-\beta),$ we have $\av{\tilde{\psi}}I=\beta\alpha+\gamma(1-\alpha)=0$
and so $\av{\tilde{\varphi}_\pm}I=x_1.$ Also,
$\av{\tilde{\psi}^2}{(0,\alpha)}=\av{\varphi_0^2}I+2\av{\varphi_0}I\beta+
\beta^2=\ve^2+\beta^2$ and we get
$\av{\tilde{\varphi}_\pm^2}I=x_1^2+(\ve^2+\beta^2)\alpha+\gamma^2(1-\alpha)=x_2.$ It remains
to calculate $\av{e^{\tilde{\varphi}_\pm}}I.$ In the notation we have introduced, equations
\eqref{t11}, \eqref{t11d} can be rewritten (for the appropriate $\delta\!$'s) as
$$
\frac{e^{-\ve/\sqrt2}}{2-e^{\ve/\sqrt2}}=\frac{1-r_1}{1-\delta}e^{r_1-\delta}=
\left(1+\frac{\delta-r_2}{\alpha(1-\delta)}\right)e^{r_1-\delta},
$$
$$
\frac{e^{\ve/\sqrt2}}{2-e^{-\ve/\sqrt2}}=\frac{1+r_1}{1+\delta}e^{-r_1+\delta}=
\left(1-\frac{\delta-r_2}{\alpha(1+\delta)}\right)e^{-r_1+\delta}.
$$
Therefore, using \eqref{4} we get
\begin{align*}
\av{e^{\tilde{\varphi}_+}}I&=e^{x_1}\av{e^{\tilde{\psi}}}I=\frac{e^{x_1+\beta-\ve/\sqrt2}}{2-e^{\ve/\sqrt2}}\,\alpha+e^{x_1+\gamma}\,(1-\alpha)\\
&=\left(1+\frac{\delta-r_2}{1-\delta}\right)\exp(x_1+r_2-\delta)=\frac{1-r_2}{1-\delta}\exp(x_1+r_2-\delta)=B^+_{\delta^+(\ve)}(x).
\end{align*}
Similarly,
$$
\av{e^{\tilde{\varphi}_-}}I=B^-_{\delta^-(\ve)}(x).
$$
We observe that $\tilde{\psi}$ (and so $\tilde{\varphi}_\pm$) does not in general belong to
$\BMO^d_\ve,$ since the jumps in the scaled function $\varphi_0$ are not at dyadic nodes for
an arbitrary $\alpha.$ We overcome this problem by constructing a rearrangement of
$\tilde{\psi}$ that belongs to $\BMO^d_\ve,$ while preserving the necessary averages. Namely,
let $\alpha_n$ be the $n\!$-th digit in the dyadic representation of $\alpha$ (we will assume
this representation is infinite, completing the sequence with zeros if needed). We define
$\psi$ as follows
\eq[7.1]{
\psi(t)=\sum_{k=1}^\infty\left\{
\alpha_k\left[\varphi_0(2^kt-1)+\beta\right]+
(1-\alpha_k)\gamma\right\}\chi_{(2^{-k},2^{-k+1})}. 
}
Naturally, we set
$$
\varphi_\pm=x_1\pm\psi.
$$
Then for any function $\mu$ we have
\begin{align*}
\av{\mu\circ\psi}I&=\sum_{k=1}^\infty\left\{\alpha_k \av{\mu(\varphi_0(2^kt-1)+\beta)}{(2^{-k},2^{-k+1})}+\mu(\gamma)(1-\alpha_k)\right\}2^{-k}\\
&=\sum_{k=1}^\infty\left\{\av{\mu(\varphi_0(t)+\beta)}I\alpha_k+\mu(\gamma)(1-\alpha_k)\right\}2^{-k}\\
&=\av{\mu\circ(\varphi_0+\beta)}I\alpha+\mu(\gamma)(1-\alpha).
\end{align*}
This calculation, with the appropriate choice of $\mu$ and the reasoning used above for
$\tilde{\varphi}_\pm,$ gives $\av{\varphi_\pm}I=x_1,$
$\av{\varphi_\pm^2}I=x_2,$ $\av{e^{\varphi_\pm}}I=B^\pm_{\delta^\pm}(x).$ It remains to check that
$\|\psi\|_{\BMO^d(I)}=\ve.$ This will immediately imply that
$\|\varphi_\pm\|_{\BMO^d(I)}=\ve.$

Take any (open) dyadic interval $J\subset I.$ We have the following trichotomy
\ben
\item
$J\subseteq (2^{-n},2^{-n+1})$ for a certain $n$ and $\alpha_n=0.$ Then
$\left.\psi\right|_J=\gamma$ and $\av{\psi^2}J-\av{\psi}J^2=0.$
\item
$J\subseteq
(2^{-n},2^{-n+1})$ for a certain $n$ and $\alpha_n=1.$ Then
$\psi(t)=\varphi_0(2^nt-1),\;\forall t\in J$ and $\av{\psi^2}J-\av{\psi}J^2\le\ve^2$ (see the
detailed consideration for $\varphi_0$ above). Also, if $J=(2^{-n},2^{-n+1}),$ then
$\av{\psi^2}J-\av{\psi}J^2=\ve^2.$
\item
$J=(0,2^{-n})$ for a certain $n.$ Then
$$
\psi(t)=
\sum_{k=n+1}^\infty\left\{ \alpha_k\left[\varphi_0(2^kt-1)+\beta\right]+
(1-\alpha_k)\gamma\right\}\chi_{(2^{-k},2^{-k+1})}
$$
So
$$
\av{\psi}J=\frac1{|J|}
\sum_{k=n+1}^\infty\left\{\alpha_k\av{\varphi_0+\beta}I+\gamma(1-\alpha_k)\right\}2^{-k}=\beta p+\gamma(1-p)
$$
and
$$
\av{\psi^2}J=(\ve^2+\beta^2)p+\gamma^2(1-p),
$$
where $p=2^n\sum_{k=n+1}^\infty \alpha_k2^{-k}.$ We have $\av{\psi^2}J-\av{\psi}J^2=
p\left[\ve^2+(\beta-\gamma)^2(1-p)\right]\df\eta(p).$ We maximize $\eta$ subject to the
constraint $0\le p\le1.$ Since
\begin{align*}
\eta'(p)&=\ve^2+(\beta-\gamma)^2(1-2p)\ge\ve^2-(\beta-\gamma)^2\\
&=\ve^2-\left(\delta-\rtde\right)^2=2\rtde\left(\delta-\rtde\right)\ge0,
\end{align*}
we have $\ds\av{\varphi^2}J-\av{\varphi}J^2\le\eta(1)=\ve^2.$ This completes the
proof of the lemma. \qedhere
\een
\end{proof}
\begin{lemd}
\label{l2d}
For every $x\in\Oe,$
\eq[t15d]{
\bel{B}^{d+}_{\ve}(x)\le
B^+_{\delta^+(\ve)}(x);~~~\bel{B}^{d-}_{\ve}(x)\ge B^-_{\delta^-(\ve)}(x)
}
\end{lemd}
\begin{proof}
We follow the template of Lemma~\ref{l2c}. As in the continuous case, we have a concavity-type
result, Lemma~\ref{l3d}, allowing us to use the induction on the order of the dyadic generation to
construct an integral sum for $\av{e^\varphi}I.$ Lemma~\ref{l4c}, the splitting lemma, cannot have a
dyadic analog, since in the dyadic setting an interval is always split in half. This lack of
splitting flexibility forces us to use a Bellman function candidate satisfying a stronger
concavity (convexity) condition. Namely, the following two inequalities are true.
\begin{lemd}
\label{l3d}
\eq[t35]{
B^+_{\delta^+(\ve)}\left(\frac12x^-+\frac12x^+\right)\ge\frac12B^+_{\delta^+(\ve)}(x^-)+\frac12B^+_{\delta^+(\ve)}(x^+)
}
\eq[t35-]{
B^-_{\delta^-(\ve)}\left(\frac12x^-+\frac12x^+\right)\le\frac12B^-_{\delta^-(\ve)}(x^-)+\frac12B^-_{\delta^-(\ve)}(x^+)
}
for any straight-line segment with the endpoints $x^\pm\in\Oe$ such that
$(x^-+x^+)/2\in\Oe.$
\end{lemd}
Assuming this lemma for the time being, take $\varphi\in \BMO^d_\ve(I).$ Observe that
$\varphi\in \BMO^d_\ve(J)$ for any dyadic subinterval $J$ of $I.$ Let $I^{0,0}=I$ and let
$I^{n,m}$ be the $m\!$-th interval of the $n\!$-th generation in the dyadic lattice based on
$I.$ Let $x^{n,m}=\left(\av{\varphi}{I^{n,m}},\av{\varphi^2}{I^{n,m}}\right).$ The argument
of Lemma~\ref{l2c} now translates verbatim to the dyadic case. For the sake of completeness we
repeat its major points. Using \eqref{t35} from Lemma~\ref{l3d} repeatedly, we get
\begin{eqnarray}
\nonumber B^+_{\delta^+(\ve)}(x^{0,0})&\ge&\ds
\frac12 B^+_{\delta^+(\ve)}(x^{1,0})+\frac12 B^+_{\delta^+(\ve)}(x^{1,1})\\
\nonumber &&\\
\label{t17d} &\ge&\ds \frac14 B^+_{\delta^+(\ve)}(x^{2,0})+ \frac14
B^+_{\delta^+(\ve)}(x^{2,1})+ \frac14 B^+_{\delta^+(\ve)}(x^{2,2})+
\frac14 B^+_{\delta^+(\ve)}(x^{2,3})\\
\nonumber &&\\
\nonumber &\ge&\ds \frac1{2^n}\sum_{m=0}^{2^n-1}B^+_{\delta^+(\ve)}(x^{n,m})=\frac1{|I|}\int_I e^{\varphi_n(s)}b_+(s_n(s))\,ds,
\end{eqnarray}
where $\varphi_n$ and $s_n$ are the same step functions that appeared in the proof of
Lemma~\ref{l2c}: $\varphi_n(s)=x_1^{n,k}$ and $s_n(s)=x_2^{n,k}-(x_1^{n,k})^2$ for $s\in I^{n,k}.$
Function $b_+$ also has a meaning similar to that in the proof of Lemma~\ref{l2c}:
$$
b_+(t)=\frac{1-\sqrt{\delta^+(\ve)^2-t}}{1-\delta^+(\ve)}\exp\left(\sqrt{\delta^+(\ve)^2-t}-\delta^+(\ve)\right).
$$
The last equality is just the statement
$B^+_{\delta^+(\ve)}(x^{n,k})=e^{\varphi_n(s)}b_+(s_n(s)),\, s\in I^{n,k}.$

Likewise, applying \eqref{t35-} repeatedly, we obtain
$$
B^-_{\delta^-(\ve)}(x^{0,0})\le
\frac1{|I^{0,0}|}\sum_{k=0}^{2^n-1}|I^{n,k}|B^-_{\delta^-(\ve)}(x^{n,k})=\frac1{|I|}\int_I
e^{\varphi_n(s)}b_-(s_n(s))\,ds.
$$
Here
$$
b_-(t)=\frac{1+\sqrt{\delta^-(\ve)^2-t}}{1+\delta^-(\ve)}\exp\left(-\sqrt{\delta^-(\ve)^2-t}+\delta^-(\ve)\right).
$$
The technical convergence arguments of Lemma~\ref{l2c} completely carry over to the dyadic case (the
quasi-Haar system in the proof of Lemma~\ref{l2c} now becomes the usual Haar system) and we obtain
$$
B^-_{\delta^-(\ve)}(\av{\psi}I,\av{\psi^2}I)\le
\frac1{|I|}\int_I e^{\psi(s)}\,ds\le
B^+_{\delta^+(\ve)}(\av{\psi}I,\av{\psi^2}I).
$$
Taking first supremum and then infimum over all $\psi\in \BMO^d_\ve(I)$ with
$\av{\psi}I=x_1$ and $\av{\psi^2}I=x_2,$ we obtain the inequalities
$$
B^+_{\delta^+(\ve)}(x)\ge \bel{B}^{d+}_{\ve}(x),\qquad B^-_{\delta^-(\ve)}(x)\le
\bel{B}^{d-}_{\ve}(x),
$$
thus proving the lemma.
\end{proof}

\noindent{\it Proof of Lemma~\ref{l3d}.}
We will first prove the ``concavity'' result for $B^+,$ i.e.
inequality \eqref{t35}, and then indicate what changes are needed in the case of $B^-.$ To
simplify notation, we will use $B,\bel{B},$ and $\delta$ without the superscript $\pm$ when
the context is unambiguous.

{\it Proof of \eqref{t35}.} We prove the inequality in the most constructive manner: for
every $\ve$ we will choose the smallest $\delta$ so that the statement of the lemma holds.
From the proof of Lemma~\ref{l1d}, it is clear that $\delta(\ve)>\ve.$

One straightforward approach would be to choose $\delta(\ve)$ large enough so that any
straight-line segment $[x^-,x^+]$ with $x^-,x^+,x^0\in\Oe$ would fit entirely inside
$\Omega_{\delta(\ve)}.$ The statement of Lemma~\ref{l3d} would then follow from Lemma~\ref{l3c}. Let us
investigate how large the $\delta(\ve)$ so chosen would be with regard to $\ve.$
\begin{prop}
\label{pr2}
If $\ve\le\frac{2\sqrt2}3\delta,$ then the segment $[x^-,x^+]$ lies entirely in $\Od,$ for
all $x^-,x^+\in\Oe$ such that $\frac12x^-+\frac12x^+\in\Oe.$
\end{prop}
\begin{proof}
We only need to consider those segments $[x^-,x^+]$ that have points outside $\Oe,$ because
otherwise $[x^-,x^+]\subset\Oe\subset\Od.$ Parameterize the points of $[x^-,x^+]$ as follows
$$
x(t)=(1-t)x^-+tx^+.
$$
Then we need to check that for the function
$$
\tau(t)=x_2(t)-x_1^2(t),\quad 0\le t\le1,
$$
the inequality $\tau(t)\le\delta^2$ holds.

Denote by $a$ and $b$ the points of intersection of the segment $[x^-,x^+]$ with the upper
boundary of $\Oe,$ the parabola $x_2=x_1^2+\ve^2.$ Since
$\frac12x^-+\frac12x^+\in\Oe,$ the segment $[a,b]$ lies between this point and one of the
endpoints $x^\pm.$ Let us call this endpoint $x^-.$ Since $\tau(t)\le\ve^2$ for $x(t)\in\Oe,$
we have
$$
\max_{x(t)\in[x^-,x^+]}\tau(t)=\max_{x(t)\in[a,b]}\tau(t).
$$
Therefore, instead of the initial segment $[x^-,x^+],$ it is sufficient to consider the shorter
segment $[a,2b-a].$ This means that without loss of generality we may assume the points $x^-$
and $\frac12x^-+\frac12x^+$ to be on the upper bound of $\Oe,$ i.~e.,
\eq[prop2-1]{
x_2^-=(x_1^-)^2+\ve^2,
}
\eq[prop2-2]{
\frac12(x_2^-+x_2^+)=\frac14(x_1^-+x_1^+)^2+\ve^2.
}
From~\eqref{prop2-1} and~\eqref{prop2-2} we get
$$
x_2^+=\frac12\bigl((x_1^+)^2-(x_1^-)^2\bigr)+x_1^-x_1^++\ve^2.
$$
Since $x^+\in\Oe,$ we have
the restriction $x_2^+\ge(x_1^+)^2,$ which is equivalent to the inequality
\eq[prop2-4]{
(x_1^-+x_1^+)^2\le2\ve^2.
}
Now, calculate $\max\tau(t)$:
\[
\begin{split}
\tau(t)&=x_2(t)-x_1^2(t)
\\
&=\bigl[(1-t)x_2^-+tx_2^+\bigr]-\bigl[(1-t)x_1^-+tx_1^+\bigr]^2
\\
&=\ve^2+\frac12(x_1^-+x_1^+)^2(t-2t^2).
\end{split}
\]
This function attains its maximum at $t=\frac14,$ so
$$
\max\tau(t)=\ve^2+\frac1{16}(x_1^-+x_1^+)^2.
$$
Taking into account inequality~\eqref{prop2-4} we get
$$
\max\tau(t)\le\frac98\ve^2\le\delta^2.
$$
This means $[x^-,x^+]\subset\Od,$ as claimed.
\end{proof}
Applying now Lemma~\ref{l3c}, we obtain 
\eq[t341]{
B_\delta\Bigl(\frac12x^-+\frac12x^+\Bigr)\ge\frac12B_\delta(x^-)+\frac12B_\delta(x^+),
}
as long as the triple $x^-,x^+,\frac12x^-+\frac12x^+\in\Od.$
We observe that if $\ve<\frac{2\sqrt2}3,$ then we can run
the machine of Lemma~\ref{l2d} to establish that 
$$
B_{\frac3{2\sqrt2}\ve}(x)\ge\bel{B}^d_{\ve}(x),~~\forall x\in\Oe. 
$$ 
Together with Lemma~\ref{l1d},
this gives us the following estimates 
\eq[t33]{ 
\frac{2\sqrt2}3\le\ve_{0}^d\le\sqrt2\log 2 
}
and 
\eq[t34]{ 
\delta(\ve)\le \frac3{2\sqrt2}\,\ve. 
} 
The rest of the $B^+$ part of the proof
of Lemma~\ref{l3d} is devoted to bridging the gap in \eqref{t33}.

So far, we have been trying to ensure that the segment $[x^-,x^+]$ lies inside the domain of
concavity of a certain function $B,$ so that we can then infer \eqref{t341}. Now, we try to
enforce that condition directly instead.

Since we are searching for $\delta(\ve)$ such that $\bel{B}^d_{\ve}=B_{\delta(\ve)},$ we
attempt to solve the extremal problem
\eq[t36]{
\begin{array}{ll}
\ds \delta(\ve)=\min_{\ve<\delta <1}
&\ds \left\{\delta:~B_{\delta}(x^0)\ge\frac12B_{\delta}(x^-)+\frac12B_{\delta}(x^+),\right.\\
&\ds \left.~\forall x^-,x^+\in\Oe~\text{such~that}~x^0=\frac12x^-+\frac12x^+\in\Oe\right\}.
\end{array}
}
We can simplify this formulation by
observing that we can, without loss of generality, set $x_1^0=0.$ Indeed, consider the change of variables
$$
\begin{array}{l}
\tilde{x}_1=x_1-x_1^0;\\
\tilde{x}_2=x_2-2x_1x_1^0+(x_1^0)^2=x_2-x_1^2+\tilde{x}_1^2.
\end{array}
$$
Then $\tilde{x}_2-\tilde{x}_1^2=x_2-x_1^2,$ i.e. the point $\tilde{x}$ belongs to $\Oe$ (or $\Od$) if and only if $x$ does. Furthermore,
condition
\eqref{t35} is equivalent to
\begin{align}
\label{t351} \notag F_\delta(x^-,x^+,x^0)&\df 2\left(1-\sqrt{\delta^2+(x^0_1)^2-x^0_2}\right)
\exp\left(\sqrt{\delta^2+(x^0_1)^2-x^0_2}\right)\\
&-\left(1-\sqrt{\delta^2+(x^-_1)^2-x^-_2}\right)\exp\left(\frac{x^-_1-x^+_1}2+\sqrt{\delta^2+(x^-_1)^2-x^-_2}\right)\\
\notag&-\left(1-\sqrt{\delta^2+(x^+_1)^2-x^+_2}\right)\exp\left(\frac{x^+_1-x^-_1}2+\sqrt{\delta^2+(x^+_1)^2-x^+_2}\right)\\
\notag&=F_\delta(\tilde{x}^-,\tilde{x}^+,0)\ge0.
\end{align}
Due to the ensuing symmetry we can also assume $x_1^+\ge0.$

Now, let
\eq[t352]{
a=\sqrt{\delta^2-x^0_2},\qquad
a_\pm=\sqrt{\delta^2+\left(x^\pm_1\right)^2-x^\pm_2},\qquad
\theta=x^+_1. 
}
Geometrically, $a$
and $a_\pm$ are the square roots of the vertical distances from $x$ and $x^\pm$ to the
parabola $x_2=x_1^2+\delta^2,$ as shown on Fig.~\ref{f13}. 
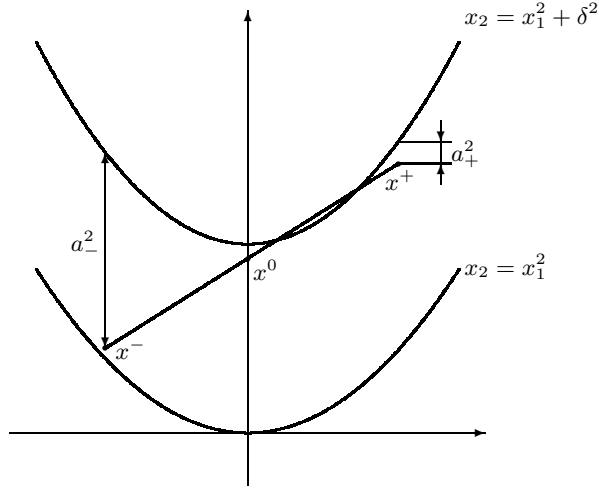
\begin{figure}[ht]
\begin{center}
\begin{picture}(200,200)
\thinlines
\put(100,0){\vector(0,1){180}}
\put(10,20){\vector(1,0){180}}
\put(46,89){\vector(0,1){37}}
\put(46,89){\vector(0,-1){37}}
%\put(157,126){\vector(0,1){4}}
%\put(157,126){\vector(0,-1){4}}
\put(157,122){\line(1,0){20}}
\put(157,130){\line(1,0){20}}
\put(173,122){\line(0,1){8}}
\put(173,138){\vector(0,-1){8}}
\put(173,114){\vector(0,1){8}}
\linethickness{.7pt}
\qbezier[1000](20,82)(100,-42)(180,82)
%\qbezier[1000](20,162)(100,15)(180,162)
\qbezier[1000](20,168)(100,15)(180,168)
\qbezier[1000](46,52)(46,52)(157,122)
%\qbezier[1000](10,41)(10,41)(190,131)
\put(46,52){\circle*{2}}
\put(100,86){\circle*{2}}
\put(157,122){\circle*{2}}
%\put(138,105){\circle*{2}}
%\put(124,98){\circle*{2}}
\put(48,48){\footnotesize $x^-$}
\put(100,78){\footnotesize $x^0$}
\put(150,112){\footnotesize $x^+$}
\put(31,89){\footnotesize $a^2_-$}
\put(175,123){\footnotesize $a^2_+$}
%\put(134,97){\footnotesize $p$}
%\put(117,102){\footnotesize $t$}
\put(180,80){\footnotesize $x_2=x_1^2$}
%\put(180,160){\footnotesize $x_2=x_1^2+\ve^2$}
\put(180,175){\footnotesize $x_2=x_1^2+\delta^2$}
%\put(190,130){\footnotesize $x_2=2ax_1+x^0_2$}
\end{picture}
\end{center}
\caption{Geometrical meaning of $a_-$ and $a_+.$}
\label{f13}
\end{figure}
Using this notation, we can rewrite the condition $F_\delta(x^-,x^+,x^0)\ge0$ as
\eq[t353]{
f_\delta(a,a_-,a_+,\theta)\df 2(1-a)e^{a}-(1-a_-)e^{-\theta+a_-}-(1-a_+)e^{\theta+a_+}\ge0
}
(we will omit the index $\delta$ when the context is clear). A straightforward calculation
shows that $a_-^2+a_+^2=2a^2+2\theta^2.$ The condition $x,x_\pm\in\Oe$ can be rewritten as
$a,a_-,a_{+}\in[\rtde,\delta]$ and the condition $x_1^+\ge0$ becomes $\theta\ge0.$ Finally,
we observe that since
$(1-u)e^{-\theta+u}+(1-v)e^{\theta+v}\ge(1-v)e^{-\theta+v}+(1-u)e^{\theta+u}$ if $0\le v\le
u$ (see Proposition~\ref{pr1}), it suffices to consider the case $a_+\le a_-$ (equivalently,
$x_2^+\ge x_2^-$) when enforcing the condition $f_\delta(a,a_-,a_+,\theta)\ge0,$ i.e. we can
consider only those segments slanted upward. We are in a position to reformulate the extremal problem \eqref{t36} as follows

For $0<\ve<\sqrt2\log2$ and $\ve<\delta<1,$ let
$$
S_{\delta,\ve}=\ds\left\{(x,y,z,w)\in[\rtde,\delta]^3\times[0,\infty);~z\le
y;~y^2+z^2=2x^2+2w^2\right\}.
$$
Then
\eq[t39]{
m(\delta,\ve)=\min\left\{f(a,a_-,a_+,\theta):(a,a_-,a_+,\theta)\in S_{\delta,\ve}\right\},
}
\eq[t40]{
\delta(\ve)=\min\{\delta:m(\delta,\ve)\ge 0\}.
}
In addition, we will need the following notation
$$
S_{\delta,\ve,a}=S_{\delta,\ve}\cap\{x=a\};~~~m_a(\delta,\ve)=\min f|_{S_{\delta,\ve,a}}.
$$
While simplifying calculations, formulation \eqref{t39},~\eqref{t40} has a drawback: the
underlying geometry of segments in $\Oe$ and/or $\Od$ is obscured. For example, the fact that
$B_\delta$ is locally concave in $\Od$ and, hence, $F_\delta\ge0$ if the whole segment
$[x^-,x^+]$ lies in $\Od,$ will take a certain amount of effort to phrase in terms of the new
variables.
\subsection{Stage 1}
We first fix $a$ and collect several geometric observations.
\begin{prop}
\label{pr3}
If $a\in [\sqrt{\delta^2-\ve^2/2},\delta],$ then
$m_a(\delta,\ve)=0.$
\end{prop}
\begin{proof}
Our assumption $a\ge \sqrt{\delta^2-\ve^2/2}$ can be reformulated as
$$
\ve^2\ge2(\delta^2-a^2)=2x_2^0=x_2^++x_2^-,
$$
hence $x_2^\pm\le\ve^2,$ so any segment $[x^-,x^+]$ with $x^-,x^+\in\Oe$ such that
$(x^-+x^+)/2=(0,\delta^2-a^2)$ lies in $\Oe.$ Therefore,
$B_\delta(x)\ge\frac12B_{\delta}(x^-)+\frac12B_{\delta}(x^+)$ or, equivalently,
$f(a,a_-,a_+,\theta)\ge0.$ Of course, if $a_-=a_+=a$ and $\theta=0,$ we have $f=0,$ which
completes the proof.
\end{proof}
\begin{obs}
\label{obs1}
If $x^-,x^+\in\Oe,$ $(x^-+x^+)/2=(0,\delta^2-a^2),$ and
$x^+_1<a+\sqrt{\delta^2-\ve^2},$ then the segment $[x^-,x^+]$ lies in $\Od.$
\end{obs}
\begin{proof}
To show this, consider the line through $x^0$ tangent to the parabola $x_2=x_1^2+\delta^2.$
The point of tangency is $t=(a,a^2+\delta^2)$ and the equation of the tangent is
\eq[t40.2]{
x_2=2ax_1+x_2^0.
}
\begin{figure}[ht]
\begin{center}
\begin{picture}(200,200)
\thinlines
\put(100,0){\vector(0,1){180}}
\put(10,20){\vector(1,0){180}}
\qbezier[1000](20,82)(100,-42)(180,82)
\qbezier[1000](20,162)(100,15)(180,162)
\qbezier[1000](20,168)(100,15)(180,168)
\qbezier[1000](46,52)(46,52)(159,123)
\qbezier[1000](10,41)(10,41)(190,131)
\put(46,52){\circle*{2}}
\put(100,86){\circle*{2}}
\put(159,123){\circle*{2}}
\put(138,105){\circle*{2}}
\put(124,98){\circle*{2}}
\put(48,48){\footnotesize $x^-$}
\put(100,78){\footnotesize $x^0$}
\put(160,122){\footnotesize $x^+$}
\put(134,97){\footnotesize $p$}
\put(117,102){\footnotesize $t$}
\put(180,80){\footnotesize $x_2=x_1^2$}
\put(180,160){\footnotesize $x_2=x_1^2+\ve^2$}
\put(180,175){\footnotesize $x_2=x_1^2+\delta^2$}
\put(190,130){\footnotesize $x_2=2ax_1+x^0_2$}
\end{picture}
\end{center}
\caption{A segment $[x^-,x^+]\nsubseteq\Od$ vs. the tangent to $x_2=x_2^2+\delta^2$}
\label{f14}
\end{figure}
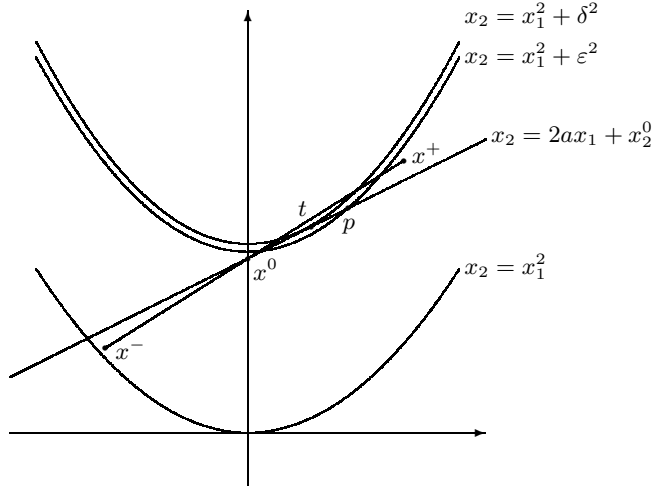
Any segment $[x^-,x^+]$ that does not lie entirely in $\Od$ will have a slope higher than
that of this tangent, see Fig.~\ref{f14}. The segment's endpoint $x^+$ will then have to be to the right of the
point $p$ of intersection of the tangent \eqref{t40.2} and the parabola $x_2=x_1^2+\ve^2,$
i.e. we will have $x^+_1>p_1.$ Solving for $p_1,$ we get
$$
p_1^2+\ve^2=2ap_1+x_2^0,
$$
so $(p_1-a)^2=\delta^2-\ve^2.$ Since $p$ is to the right of $t,$ we have $p_1=a+\rtde,$
completing our observation.
\end{proof}
We now show that the only ``interesting'' (i.e. not obviously non-negative) minimum of $f$ can
happen at the ``corner'' $a_-=\delta,a_+=\rtde, \theta^2+a^2=\delta^2-\ve^2/2.$ More
precisely, we have the following proposition.
\begin{prop}
\label{pr4}
If $a\in[\rtde,\sqrt{\delta^2-\ve^2/2}),$ then
$$
m_a(\delta,\ve)=\min\{0,f(a,\delta,\rtde,\sqrt{\delta^2-\ve^2/2-a^2})\}.
$$
\end{prop}
\begin{proof}
Fix an $a\in[\rtde,\sqrt{\delta^2-\ve^2/2}).$ $S_{\delta,\ve,a}$ is the portion
of the hyperboloid $a_-^2+a_+^2=2\theta^2+2a^2$ sitting above this ``quadrilateral'' region
in the $(a_-,a_+)\!$-plane (the plane $\theta=0$). Fig.~\ref{p} shows this region, while Fig.~\ref{p1} gives the corresponding region in the original variables.
\begin{figure}[ht]
\begin{center}
\begin{picture}(100,100)
\thinlines
\put(-10,0){\vector(1,0){115}}
\put(0,-10){\vector(0,1){115}}
\put(15,-2){\line(0,1){4}}
\put(95,-2){\line(0,1){4}}
\put(-2,15){\line(1,0){4}}
\put(-2,95){\line(1,0){4}}
\thicklines
\qbezier(45,15)(45,15)(95,15)
\qbezier(95,15)(95,15)(95,95)
\qbezier(95,95)(95,95)(35,35)
\qbezier(35,35)(45,30)(45,15)
\put(0,-11){\footnotesize $\rtde$}
\put(-43,15){\footnotesize $\rtde$}
\put(90,-11){\footnotesize $\delta$}
\put(-15,90){\footnotesize $\delta$}
\put(105,-2){\footnotesize $a_-$}
\put(-5,106){\footnotesize $a_+$}
\put(65,9){\scriptsize $e_1$}
\put(97,50){\scriptsize $e_2$}
\put(50,65){\scriptsize $e_3$}
\put(30,20){\scriptsize $e_4$}
\end{picture}
\end{center}
\caption{The projection of $S_{\delta,\ve,a}$ onto the $(a_-,a_+)\!$-plane.}
\label{p}
\end{figure}
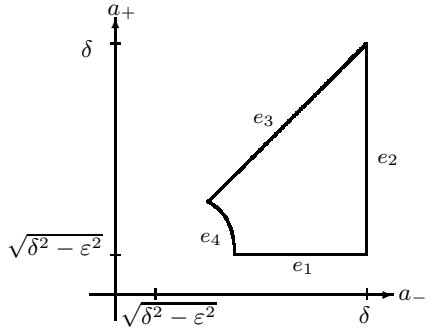
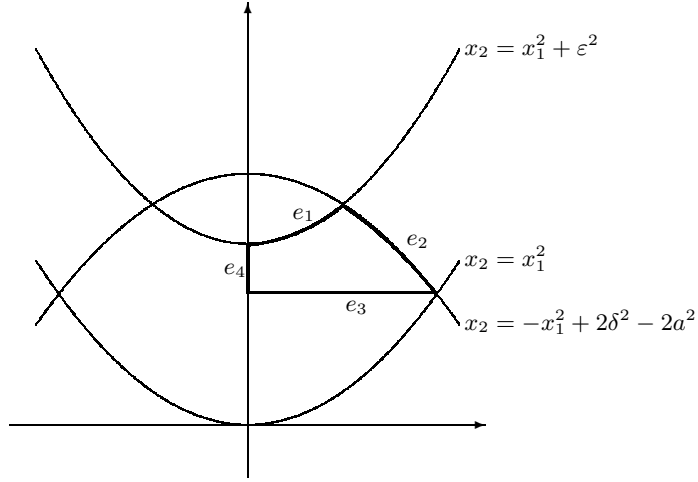
\begin{figure}[ht]
\begin{center}
\begin{picture}(200,200)
\thinlines \put(100,0){\vector(0,1){180}} \put(10,20){\vector(1,0){180}}
\qbezier(20,82)(100,-42)(180,82) \qbezier(20,162)(100,15)(180,162)
\qbezier(20,58)(100,172)(180,58) \thicklines \qbezier(100,70)(100,70)(171,70)
\qbezier(100,70)(100,70)(100,88) \qbezier(100,88)(120,90)(136,103)
\qbezier(136,103)(155,90)(171,70) \put(180,80){\footnotesize $x_2=x_1^2$}
\put(180,160){\footnotesize $x_2=x_1^2+\ve^2$}
\put(180,56){\footnotesize$x_2=-x_1^2+2\delta^2-2a^2$} \put(89,77){\footnotesize $e_4$}
\put(135,63){\footnotesize $e_3$} \put(158,88){\footnotesize $e_2$}
\put(115,98){\footnotesize $e_1$}
\end{picture}
\end{center}
\caption{The domain of variation of $x^+$ in $\Oe$ corresponding to $S_{\delta,\ve,a}.$}
\label{p1}
\end{figure}
The edges are as follows: $e_1\!\!:~a_+=\rtde,$ $e_2\!\!:~a_-=\delta,$ $e_3\!\!:~a_+=a_-,$
$e_4\!\!:~a_-^2+a_+^2=2a^2.$ The fact that the $(e_1,e_2)$ corner is in the picture is due to
the condition $a<\sqrt{\delta^2-\ve^2/2}.$ We include the degenerate cases $a=\rtde,$
$a=\sqrt{\delta^2-\ve^2/2}$ when edges $e_4$ and $e_1,$ respectively, shrink to a point, in
the general computation.

To minimize $f$ on $S_{\delta,\ve,a},$ we will utilize Lagrange multipliers in the interior
of the quadrilateral as well as on its nontrivial edges $e_1$ and $e_2.$

{\bf Interior.} We form the corresponding Lagrangian:
$$
L(a_-,a_+,\theta,\lambda)=2(1-a)e^a-(1-a_-)e^{a_--\theta}-(1-a_+)e^{a_++\theta}-\lambda(a_-^2+a_+^2-2\theta^2-2a^2).
$$
$\nabla L=0$ yields
$$
\begin{array}{rcl}
a_-e^{a_--\theta}&=&2\lambda a_-\\
a_+e^{a_++\theta}&=&2\lambda a_+\\
(1-a_-)e^{a_--\theta}-(1-a_+)e^{a_++\theta}&=&-4\theta\lambda\\
a_-^2+a_+^2&=&2\theta^2+2a^2
\end{array}
$$
The first two equations give $a_+=a_--2\theta.$ Plugging this into the last equation, we
obtain $(a_--\theta)^2=a^2,$ hence $a_-=a+\theta$ ($a_-=-a+\theta$ would imply
$a_+=-a-\theta<0,$ an impossibility). Calculating $f$ for this combination of variables, we
obtain
$$
f(a,a_-,a_+,\theta)=2(1-a)e^a-(1-a-\theta)e^a-(1-a+\theta)e^a=0.
$$

{\bf Edge $e_1.$} We have $a_+=\rtde,$ so $a_-^2+a_+^2=2\theta^2+2a^2$ becomes
$a_-^2+\delta^2-\ve^2=2\theta^2+2a^2.$ Again, we form the Lagrangian:
$$
l(a_-,\theta,\lambda)=2(1-a)e^a-(1-a_-)e^{a_--\theta}-(1-\rtde)e^{\rtde+\theta}-\lambda(a_-^2+\delta^2-\ve^2-2\theta^2-2a^2).
$$
$\nabla l=0$ yields
$$
\begin{array}{rcl}
a_-e^{a_--\theta}&=&2\lambda a_-\\
(1-a_-)e^{a_--\theta}-(1-\rtde)e^{\rtde+\theta}&=&-4\theta\lambda\\
a_-^2+\delta^2-\ve^2&=&2\theta^2+2a^2
\end{array}
$$
The first two equations give $(1-\rtde)e^{\rtde+\theta}=(1-a_-)e^{a_--\theta}+2\theta e^{a_--\theta}$ and so
$$
f(a,a_-,a_+,\theta)=2\left((1-a)e^a-(1-(a_--\theta))e^{a_--\theta}\right).
$$
Two separate cases need to be considered here. If $\theta<a+\rtde,$ then, by Observation~\ref{obs1},
the whole segment $[x_-,x_+],$ underlying our $a,\theta$ notation, lies inside $\Od.$ But
$B_\delta$ is locally concave inside $\Od,$ so
$2B_\delta(x^0)-B_\delta(x^-)-B_\delta(x^+)\ge0,$ which is equivalent to
$f(a,a_-,a_+,\theta)\ge0.$

If, on the other hand, $\theta\ge a+\rtde,$ then $\theta^2+a^2-2\theta a\ge\delta^2-\ve^2,$
hence
$$
a_-^2+\theta^2+a^2-2\theta a\ge a_-^2+\delta^2-\ve^2=2\theta^2+2a^2,
$$ so
$$
a_-^2\ge(\theta+a)^2\Longrightarrow a_--\theta\ge a.
$$
Using Proposition~\ref{pr1}, we obtain $f\ge0.$

{\bf Edge $e_2.$} We have $a_-=\delta,$ so $a_-^2+a_+^2=2\theta^2+2a^2$ becomes
$a_+^2+\delta^2=2\theta^2+2a^2.$ Once more, we form the Lagrangian:
$$
l(a_+,\theta,\lambda)=
2(1-a)e^a-(1-\delta)e^{\delta-\theta}-(1-a_+)e^{a_++\theta}-\lambda(a_+^2+\delta^2-2\theta^2-2a^2).
$$
$\nabla l=0$ gives
$$
\begin{array}{rcl}
a_+e^{a_++\theta}&=&2\lambda a_+\\
(1-\delta)e^{\delta-\theta}-(1-a_+)e^{a_++\theta}&=&-4\theta\lambda\\
a_+^2+\delta^2&=&2\theta^2+2a^2
\end{array}
$$
The first two equations give $(1-\delta)e^{\delta-\theta}=(1-a_+)e^{a_++\theta}-2\theta
e^{a_++\theta},$ so $(1-\delta)e^\delta=(1-(a_++2\theta))e^{a_++2\theta}$ and so
$a_++2\theta=\delta.$ The third equation then gives $a=\delta-\theta$ and we have
$$
f(a,a_-,a_+,\theta)=f(\delta-\theta,\delta,\delta-2\theta,\theta)=0.
$$

{\bf Edge $e_3.$} If $a_-=a_+,$ the underlying segment $[x^-,x^+]$ is horizontal and thus
lies entirely in $\Oe.$ In this case, $f\ge0.$

{\bf Edge $e_4.$} If $a_-^2+a_+^2=2a^2,$ then $\theta=0$ and we get a vertical segment, also
lying entirely in $\Oe.$

{\bf Vertices.} The only nontrivial vertex is $a_-=\delta,a_+=\rtde.$ If we make sure that
$f\ge0$ at this vertex, then we will have $f\ge0$ on $S_{\delta,\ve,a}.$ This completes the
proof of the proposition.
\end{proof}
In our search of a segment that would minimize $f$ on $S_{\delta,\ve,a},$ we have now planted
the endpoints $x^+$ and $x^-$ on the top and bottom boundary of $\Oe,$ correspondingly. To
finish the proof of Lemma~\ref{l3d}, we need to vary $x^0.$ Another geometric observation is in
order.
\begin{obs}
\label{obs2}
If $\theta\le(\delta+\rtde)/2,$ then
$f(\sqrt{\delta^2-\ve^2/2-\theta^2},\delta,\rtde,\theta)\ge0.$
\end{obs}
\begin{proof}
We demonstrate this by rephrasing Observation 1. Namely, we investigate what the condition
$\theta\le a+\rtde$ means when $a_-=\delta$ and $a_+=\rtde.$

Since $a_-^2+a_+^2=2\theta^2+2a^2,$ we have $a=\sqrt{\delta^2-\ve^2/2-\theta^2}.$ Therefore,
the condition becomes
$$
\theta\le\sqrt{\delta^2-\ve^2/2-\theta^2}+\rtde.
$$
If $\theta\le\rtde,$ Observation~\ref{obs1} works and $f\ge0.$ If $\theta\ge\rtde,$ the above
inequality is equivalent to
$$
\theta^2-\theta\rtde\le\frac{\ve^2}4.
$$
We continue
$$
\left(\theta-\frac{\rtde}2\right)^2\le\frac{\delta^2}4,
$$
which gives (taking into account the fact that $\theta\ge0$)
$$
0\le\theta\le\frac{\delta+\rtde}2.\mbox{\qedhere}
$$
\end{proof}
We are now in a position to finalize the first stage of the extremal problem \eqref{t39},
\eqref{t40}. 
\begin{prop}
\label{pr5}
$$
m(\delta,\ve)=\min\{0,f(\rtde,\delta,\rtde,\ve/\sqrt2)\}.
$$
\end{prop}
\begin{proof}
By Propositions~\ref{pr3} and \ref{pr4} we have
$$
m(\delta,\ve)=\min_{\rtde\le a\le\delta}m_a(\delta,\ve)=
\min\{0,\min_{\rtde\le
a\le\sqrt{\delta^2-\ve^2/2}}f(a,\delta,\rtde,\sqrt{\delta^2-\ve^2/2-a^2})\}.
$$
Expressing, as has been our custom, everything in terms of
$\theta=\sqrt{\delta^2-\ve^2/2-a^2},$ we set out to minimize the function
$$
V(\theta)\df f(\sqrt{\delta^2-\ve^2/2-\theta^2},\delta,\rtde,\theta),\qquad
0\le\theta\le\frac{\ve}{\sqrt2}\,.
$$
The interval $[0,\ve/\sqrt2]$ for $\theta$ is determined from the condition
$a_-^2+a_+^2=2a^2+2\theta^2,a\ge\rtde.$ Geometrically, we are sliding $x^0$ upward, while
$x^+$ and $x^-$ slide along the top and bottom boundary curves of $\Oe.$ We have
$$
V(\theta)=2(1-\sqrt{\delta^2-\ve^2/2-\theta^2})e^{\sqrt{\delta^2-\ve^2/2-\theta^2}}-(1-\rtde)e^{\rtde+\theta}-(1-\delta)e^{\delta-\theta}.
$$
Assume that $V$ has a local extremum $\theta=\theta_*$ in the interval $(0,\ve/\sqrt2).$ Then
$V'(\theta_*)=0,$ i.e.
$$
2\theta_*e^{\rtdet}-(1-\rtde)e^{\rtde+\theta_*}+(1-\delta)e^{\delta-\theta_*}=0.
$$
We have
$$
(1-\rtde)e^{\rtde+\theta_*}=2\theta_*e^{\rtdet}+(1-\delta)e^{\delta-\theta_*}
$$
and so
$$
V(\theta_*)=2e^{-\theta_*}\left[(1-(\rtdet+\theta_*))e^{\rtdet+\theta_*}-(1-\delta)e^{\delta}\right].
$$
If $0\le\theta_*\le(\delta+\rtde)/2,$ then, by Observation 2, $V(\theta_*)\ge0.$ What happens
if $(\delta+\rtde)/2\le\theta_*\le\ve/\sqrt2?$ First of all, in order to ensure that this
question makes sense, we observe that the inequality $(\delta+\rtde)/2\le\ve/\sqrt2$ is
equivalent to the condition $\delta\le\frac3{2\sqrt2}\ve.$ If it does not hold, Proposition 2
implies that $V(\theta_*)\ge0.$ Assuming the inequality does hold, we have
$$
\frac{\delta+\rtde}2\le\theta_*,
$$
which, after rearrangement and squaring, becomes
$$
\delta^2-\ve^2\le4\theta_*^2+\delta^2-4\theta_*\delta,
$$
then
$$
\delta^2-\ve^2/2-\theta_*^2\le\theta_*^2+\delta^2-2\theta_*\delta
$$
and, finally (since $\theta_*\le\frac\ve{\sqrt2}<\delta$),
$$
\delta\ge\theta_*+\rtdet.
$$
It follows from Proposition~\ref{pr1} that $V(\theta_*)\ge0.$

This consideration means that
$$
\min_{0\le\theta}\{0,V(\theta)\}= \min\{0,V(0),V(\ve/\sqrt2)\},
$$
but $V(0)\ge0$ (by Observation~\ref{obs2}) and therefore
$$
m(\delta,\ve)=\min\{0,V(\ve/\sqrt2)\}=
\min\{0,f(\rtde,\delta,\rtde,\ve/\sqrt2)\}.\mbox{\qedhere}
$$
\end{proof}
We have completed the first stage of our extremal problem. We can now rephrase \eqref{t40}, as follows. Let
$$
g(\delta,\ve)=f(\rtde,\delta,\rtde,\ve/\sqrt2).
$$
Equivalently,
$$
g(\delta,\ve)=(1-\rtde)e^{\rtde}\left(2-e^{\ve/\sqrt2}\right)-(1-\delta)e^{\delta-\ve/\sqrt2}.
$$
Then
\eq[t43]{
\delta(\ve)=\min_{\ve<\delta<1}\{\delta:g(\delta,\ve)\ge 0\}.
}
\subsection{Stage 2}
The following simple result will complete the ``+'' part of the proof of Lemma~\ref{l3d}.

\begin{prop}
\label{pr6}
For any $\ve,$
$0<\ve<\sqrt2\log2,$ the equation $g(\delta,\ve)=0$ has a unique solution on
the interval $(\ve,1)$ and it is $\delta(\ve)$ from~{\rm\eqref{t43}}.
\end{prop}
\begin{proof}
Differentiating $g$ with respect to $\delta,$ we obtain
$$
\frac{\partial g}{\partial \delta}(\delta,\ve)=\delta\left[e^{\delta-\ve/\sqrt2}-e^{\rtde}\left(2-e^{\ve/\sqrt2}\right)\right].
$$
If $\ve<\delta<\frac3{2\sqrt2}\ve,$ then $\delta-\ve/\sqrt2>\rtde,$ we have
$$
\frac{\partial g}{\partial \delta}(\delta,\ve)\ge\delta e^{\rtde}\left(-1+e^{\ve/\sqrt2}\right)>0.
$$
If $\ve<\frac{2\sqrt2}3$ and $\frac3{2\sqrt2}\ve\le\delta<1,$ we know that $g(\delta,\ve)>0.$
Hence, if the equation $g(\delta,\ve)=0$ has a root on the interval $(\ve,1),$ the root is
unique. Therefore, to prove the proposition, it suffices to show that the equation
$g(\delta,\ve)=0$ has a solution on the interval $(\ve,1).$ To do this, we check that
$g(\ve,\ve)<0$ and $g(1,\ve)>0.$

At the left endpoint,
\begin{align*}
g(\ve,\ve)&=2-e^{\ve/\sqrt2}-(1-\ve)e^{\ve-\ve/\sqrt2}=
e^{-\ve/\sqrt2}\left(2e^{\ve/\sqrt2}-e^{\sqrt2\ve}-(1-\ve)e^{\ve}\right)\\
&=e^{-\ve/\sqrt2}\sum_{k=0}^{\infty}\frac{\ve^k}{k!}\left[\frac2{2^{k/2}}-1-2^{k/2}+k\right]
=e^{-\ve/\sqrt2}\sum_{k=3}^{\infty}\frac{\ve^k}{k!}\left[\frac2{2^{k/2}}-1-2^{k/2}+k\right].
\end{align*}
If $k=3,$ we get
$\frac2{2^{k/2}}-1-2^{k/2}+k=\frac1{\sqrt2}-2\sqrt2+2=\frac{2\sqrt2-3}{\sqrt2}<0.$ If
$k\ge4,$ we have $2^{1-k/2}<1$ and $k\le2^{k/2},$ so $\frac2{2^{k/2}}-1-2^{k/2}+k<0.$ We
conclude that $g(\ve,\ve)<0.$

At the other endpoint we have
$$
g(1,\ve)=\left(1-\sqrt{1-\ve^2}\right)e^{\sqrt{1-\ve^2}}\left(2-e^{\ve/\sqrt2}\right).
$$
Since $\ve<\sqrt2\log2,$ we conclude that $g(1,\ve)>0.$  This completes the proof of the
proposition and the $B^+$ part of Lemma~\ref{l3d}.
\end{proof}
{\it Proof of \eqref{t35-}.} We briefly outline what changes are necessary in the preceding
to prove the second half of Lemma~\ref{l3d}. We will designate the analogs of the propositions and
observations with a ``$-$'' sign. Proposition~\ref{pr2} implies that if
$\ve\le\frac{2\sqrt2}3\delta,$ then
$$
B^-_\delta\Bigl(\frac12x^-+\frac12x^+\Bigr)\le\frac12B^-_\delta(x^-)+\frac12B^-_\delta(x^+),
$$
for all $x^-,x^+\in\Oe$ such that $\frac12x^-+\frac12x^+\in\Oe.$

The argument allowing us to consider only those line segments with $x^-_1+x^+_1=0$ still
works. However, there is an important difference in the case of $B^-:$ we now consider those
segments slanted downward, i.e. those whose right endpoint $x^+$ is lower that the left
endpoint $x^-$ (recall that previously we considered only those slanted upward). Next, we
formulate the two-stage extremal problem for $B^-.$ As in \eqref{t352}, let
$$
a=\sqrt{\delta^2-x^0_2},\qquad
a_\pm=\sqrt{\delta^2+\left(x^\pm_1\right)^2-x^\pm_2},\qquad
\theta=x^+_1.
$$
Also, let
$$
f^-_\delta(a,a_-,a_+,\theta)\df 2(1+a)e^{-a}-(1+a_-)e^{-\theta-a_-}-(1+a_+)e^{\theta-a_+}.
$$
Observe that $f^-_\delta(a,a_-,a_+,\theta)=f_\delta(-a,-a_-,-a_+,\theta),$ where $f_\delta$
is defined by \eqref{t353}. We will mimic the formulation \eqref{t39}, \eqref{t40}, but
designate key ingredients with a ``$-$'' to avoid confusion and facilitate cross-reference.

For $0<\ve<\delta,$ let
$$
S^-_{\delta,\ve}=\ds\left\{(x,y,z,w)\in[\rtde,\delta]^3\times[0,\infty);~z\ge
y;~y^2+z^2=2x^2+2w^2\right\}.
$$
Then \eq[t39-]{ m^-(\delta,\ve)=\max\left\{f^-(a,a_-,a_+,\theta)\colon (a,a_-,a_+,\theta)\in
S^-_{\delta,\ve}\right\}, } \eq[t40-]{ \delta^-(\ve)=\min\{\delta\colon m^-(\delta,\ve)\le
0\}. } As before, we will need the following notation
$$
S^-_{\delta,\ve,a}=S^-_{\delta,\ve}\cap\{x=a\};\qquad m^-_a(\delta,\ve)=\max
f^-|_{S^-_{\delta,\ve,a}}.
$$
\subsection{Stage 1$^-$}
Again, we fix $a$ and collect several geometric facts. The first one is identical in meaning
and proof to Proposition~\ref{pr3}, stating that if the midpoint $x^0$ is low enough, then the whole
segment $[x^-,x^+]$ is inside $\Oe.$
\setcounter{propm}{2}
\begin{propm}
\label{pr3-}
If $a\in [\sqrt{\delta^2-\ve^2/2},\delta],$ then $m^-_a(\delta,\ve)=0.$
\end{propm}
We now state the following analog of the key Proposition\ref{pr4}.
\begin{propm}
\label{pr4-}
If $a\in[\rtde,\sqrt{\delta^2-\ve^2/2}),$ then
$$
m^-_a(\delta,\ve)=\max\{0,f^-(a,\rtde,\delta,\sqrt{\delta^2-\ve^2/2-a^2})\}.
$$
\end{propm}
\begin{proof}
As before, fix an $a\in[\rtde,\sqrt{\delta^2-\ve^2/2}).$ We have a picture for
$S^-_{\delta,\ve,a},$ Fig.~\ref{p-}, which is a reflection of the corresponding picture for
$S_{\delta,\ve,a}$ on Fig. \ref{p} in the line $a_+=a_-.$
\begin{figure}
\begin{center}
\begin{picture}(100,100)
\thinlines
\put(-10,0){\vector(1,0){115}}
\put(0,-10){\vector(0,1){115}}
\put(15,-2){\line(0,1){4}}
\put(95,-2){\line(0,1){4}}
\put(-2,15){\line(1,0){4}}
\put(-2,95){\line(1,0){4}}
%\linethickness{.4pt}
\thicklines
\qbezier(15,45)(15,45)(15,95)
\qbezier(15,95)(15,95)(95,95)
\qbezier(95,95)(95,95)(35,35)
\qbezier(35,35)(30,45)(15,45)
\put(0,-11){\footnotesize $\rtde$}
\put(-43,15){\footnotesize $\rtde$}
\put(90,-11){\footnotesize $\delta$}
\put(-15,90){\footnotesize $\delta$}
\put(105,-2){\footnotesize $a_-$}
\put(-5,106){\footnotesize $a_+$}
\put(5,65){\scriptsize $e_1$}
\put(50,99){\scriptsize $e_2$}
\put(64,55){\scriptsize $e_3$}
\put(19,35){\scriptsize $e_4$}
\end{picture}
\end{center}
\caption{The projection of $S^-_{\delta,\ve,a}$ onto the $(a_-,a_+)\!$-plane.} 
\label{p-}
\end{figure}
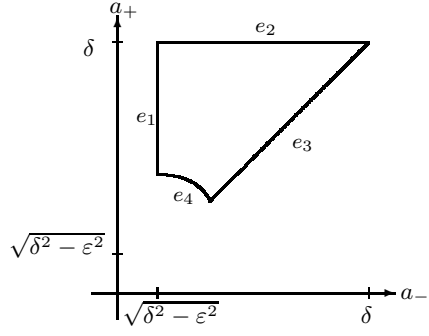
The edges are: $e_1\!\!:~a_-=\rtde,$ $e_2\!\!:~a_+=\delta,$ $e_3\!\!:~a_+=a_-,$
$e_4\!\!:~a_-^2+a_+^2=2a^2.$ Again, we make ample use of Lagrange multipliers.

{\bf Interior.} We form the corresponding Lagrangian
$$
L(a_-,a_+,\theta,\lambda)=2(1+a)e^{-a}-(1+a_-)e^{-a_--\theta}-(1+a_+)e^{-a_++\theta}-\lambda(a_-^2+a_+^2-2\theta^2-2a^2).
$$
$\nabla L=0$ yields
$$
\begin{array}{rcl}
a_-e^{-a_--\theta}&=&2\lambda a_-\\
a_+e^{-a_++\theta}&=&2\lambda a_+\\
(1+a_-)e^{-a_--\theta}-(1+a_+)e^{-a_++\theta}&=&-4\theta\lambda\\
a_-^2+a_+^2&=&2\theta^2+2a^2
\end{array}
$$
The first two equations give $a_+=a_-+2\theta.$ Plugging this into the last equation, we
obtain $(a_-+\theta)^2=a^2;$ the only acceptable solution is $a_-=a-\theta.$ Calculating
$f^-$ for this combination of variables, we obtain $f^-=0.$

{\bf Edge $e_1.$} We have $a_-=\rtde,$ so $a_-^2+a_+^2=2\theta^2+2a^2$ becomes
$a_+^2+\delta^2-\ve^2=2\theta^2+2a^2.$ Again, we form the Lagrangian:
$$
l(a_+,\theta,\lambda)=2(1+a)e^{-a}-(1+a_+)e^{-a_++\theta}-(1+\rtde)e^{-\rtde-\theta}-\lambda(a_+^2+\delta^2-\ve^2-2\theta^2-2a^2).
$$
$\nabla l=0$ yields
$$
\begin{array}{rcl}
a_+e^{-a_++\theta}&=&2\lambda a_+\\
-(1+a_+)e^{-a_++\theta}+(1+\rtde)e^{-\rtde-\theta}&=&-4\theta\lambda\\
a_+^2+\delta^2-\ve^2&=&2\theta^2+2a^2
\end{array}
$$
The first two equations give $(1+\rtde)e^{-\rtde-\theta}=(1+a_+)e^{-a_++\theta}-2\theta e^{-a_++\theta}$ and so
$$
(1+\rtde)e^{-\rtde}=(1+(a_+-2\theta))e^{-(a_+-2\theta)}.
$$
If $a_+\ge2\theta,$ then the we have $a_+-2\theta=\rtde.$ Plugging this into the third
equation, we get $a_+=2a-\rtde,$ $\theta=a-\rtde.$ With these values,
$$
f^-(a,a_-,a_+,\theta)=f^-(a,\rtde,2a-\rtde,a-\rtde)=0.
$$
If $a_+<2\theta,$ $2\theta-a_+<\rtde$ (the negative solution of the equation $(1+t)e^{-t}=c,$
$0<c<1,$ is always smaller in absolute value than the positive one). So
$a_+^2>4\theta^2+\delta^2-\ve^2-4\theta\rtde$ and
$2\theta^2+2a^2=a_+^2+\delta^2-\ve^2>4\theta^2+2(\delta^2-\ve^2)-4\theta\rtde.$ This gives
$a>\theta-\rtde.$ By Observation~\ref{obs1}, $f^-\le0.$

{\bf Edge $e_2.$} We have $a_+=\delta,$ so $a_-^2+a_+^2=2\theta^2+2a^2$ becomes
$a_-^2+\delta^2=2\theta^2+2a^2.$ Once more, we form the Lagrangian:
$$
l(a_-,\theta,\lambda)=
2(1+a)e^{-a}-(1+\delta)e^{-\delta+\theta}-(1+a_-)e^{-a_--\theta}-\lambda(a_-^2+\delta^2-2\theta^2-2a^2).
$$
$\nabla l=0$ gives
$$
\begin{array}{rcl}
a_-e^{-a_--\theta}&=&2\lambda a_-\\
-(1+\delta)e^{-\delta+\theta}+(1+a_-)e^{-a_--\theta}&=&-4\theta\lambda\\
a_-^2+\delta^2&=&2\theta^2+2a^2
\end{array}
$$
The first two equations give $(1+\delta)e^{-\delta+\theta}=(1+a_-)e^{-a_--\theta}+2\theta
e^{-a_--\theta},$ so
$$
(1+\delta)e^{-\delta}=(1+(a_-+2\theta))e^{-(a_-+2\theta)},
$$
which gives $a_-+2\theta=\delta.$ Plugging this into the third equation, we obtain
$a_-=2a-\delta,$ $\theta=\delta-a.$ With these values,
$$
f^-(a,a_-,a_+,\theta)=f^-(a,2a-\delta,\delta,\delta-a)=0.
$$
As before, edges $e_3$ and $e_4$ are trivial and the only nontrivial vertex is
$(a_-,a_+)=(\rtde,\delta).$ This consideration completes the proof of
Proposition~$4^-.$
\end{proof}

We have the appropriate analog of Observation~\ref{obs2} in terms of the function $f^-.$ 
\setcounter{obsm}{1}
\begin{obsm}
\label{obs2-}
If $\theta\le(\delta+\rtde)/2,$ then
$f^-(\sqrt{\delta^2-\ve^2/2-\theta^2},\rtde,\delta,\theta)\le0.$
\end{obsm}
To complete this stage of our program, we need
\begin{propm}
\label{pr5-}
$$
m^-(\delta,\ve)=\max\{0,f^-(\rtde,\rtde,\delta,\ve/\sqrt2)\}.
$$
\end{propm}
\begin{proof}
By Proposition~\ref{pr3} and Proposition~\ref{pr4} we have
$$
m^-(\delta,\ve)=\max_{\rtde\le a\le\delta}m^-_a(\delta,\ve)= \max\{0,\max_{\rtde\le
a\le\sqrt{\delta^2-\ve^2/2}}f^-(a,\rtde,\delta,\sqrt{\delta^2-\ve^2/2-a^2})\}.
$$
Similarly to the ``$+$'' case, we express everything in terms of $\theta$ and maximize the
function
$$
V^-(\theta)\df f^-(\sqrt{\delta^2-\ve^2/2-\theta^2},\rtde,\delta,\theta),\qquad
0\le\theta\le\frac{\ve}{\sqrt2}\,.
$$
We have
$$
V^-(\theta)=2(1+\sqrt{\delta^2-\ve^2/2-\theta^2})e^{-\sqrt{\delta^2-\ve^2/2-\theta^2}}-
(1+\rtde)e^{-\rtde-\theta}-(1+\delta)e^{-\delta+\theta}.
$$
Assume that $V^-$ has a local extremum $\theta=\theta_*$ in the interval $(0,\ve/\sqrt2).$
Then $(V^-)'(\theta_*)=0,$ i.e.
$$
2\theta_*e^{-\rtdet}+(1+\rtde)e^{-\rtde-\theta_*}-(1+\delta)e^{-\delta+\theta_*}=0.
$$
Solving for $(1+\delta)e^{-\delta+\theta_*}$ and plugging the result into the expression for
$V^-,$ we get
$$
V^-(\theta_*)=2e^{-\theta_*}\left[(1+(\rtdet-\theta_*))e^{-(\rtdet-\theta_*)}-(1+\rtde)e^{-\rtde}\right].
$$
If $0\le\theta_*\le(\delta+\rtde)/2,$ then $V^-(\theta_*)\le0$ by Observation~\ref{obs2-}. Assume
now that $(\delta+\rtde)/2\le\theta_*\le\ve/\sqrt2.$ Since the function
$\theta_*\mapsto\theta_*-\rtdet$ is increasing in $\theta_*$, it attains its minimum at the
left endpoint
$$
\theta_*=\frac{\delta+\rtde}2,
$$
and this minimum is $\rtde$, i.~e.
$$
\theta_*-\rtdet\ge\rtde.
$$
Since $(1-t_2)e^{t_2}\le(1+t_1)e^{-t_1}$ for $0\le t_1\le t_2,$ we have
\begin{align*}
(1+&(\rtdet-\theta_*))e^{-(\rtdet-\theta_*)}-(1+\rtde)e^{-\rtde}\\
&=(1-|\rtdet-\theta_*|)e^{|\rtdet-\theta_*|}-(1+\rtde)e^{-\rtde}\le0
\end{align*}
and so $V^-(\theta_*)\le0.$ This means that
$$
\max_{0\le\theta}\{0,V^-(\theta)\}= \max\{0,V^-(0),V^-(\ve/\sqrt2)\}.
$$
But $V^-(0)\le0$ (by Observation~\ref{obs2-}) and therefore
$$
m^-(\delta,\ve)=\max\{0,V^-(\ve/\sqrt2)\}= \max\{0,f^-(\rtde,\rtde,\delta,\ve/\sqrt2)\}.\mbox{\qedhere}
$$
\end{proof}
This completes Stage 1$^-.$ We rephrase \eqref{t40-} by analogy with the ``$+$'' case. Let
$$
g^-(\delta,\ve)=f^-(\rtde,\rtde,\delta,\ve/\sqrt2).
$$
Equivalently, 
$$
g^-(\delta,\ve)=(1+\rtde)e^{-\rtde}\left(2-e^{-\ve/\sqrt2}\right)-(1+\delta)e^{-\delta+\ve/\sqrt2}.
$$
Then \eq[t43-]{ \delta^-(\ve)=\min_{\ve<\delta}\{\delta:g^-(\delta,\ve)\le 0\}. }
\subsection{Stage 2$^-$}
The following proposition will complete the proof of Lemma~\ref{l3d}.
\begin{propm}
\label{pr6-}
For any $\ve>0$ the equation $g^-(\delta,\ve)=0$ has a unique solution on the interval
$\bigl(\ve,\frac3{2\sqrt2}\ve\bigr)$ and it is $\delta^-(\ve)$ from~{\rm\eqref{t43-}}.
\end{propm}
\begin{proof}
At the left endpoint, we have
$$
g^-(\ve,\ve)=2-e^{-\ve/\sqrt2}-(1+\ve)e^{-\ve+\ve/\sqrt2}.
$$
Then, after differentiating and rearrangement,
$$
[g^-(\ve,\ve)]'=\frac1{\sqrt2}e^{-\ve+\ve/\sqrt2}\left[(\sqrt2-1)\ve-1+e^{-(\sqrt2-1)\ve}\right]>0,
$$
since $x>1-e^{-x}$ for $x>0.$ Making use of the fact that $g^-(0,0)=0,$ we get
$g^-(\ve,\ve)>0,$ $\forall\ve>0.$

On the other hand,
$$
g^-\Bigl(\frac3{2\sqrt2}\ve,\ve\Bigr)=e^{-\frac{\ve}{2\sqrt2}}\left(1-\frac{\ve}{2\sqrt2}-
\Bigl(1+\frac{\ve}{2\sqrt2}\Bigr)e^{-\frac{\ve}{\sqrt2}}\right)<0,\qquad\forall\ve\le0.
$$ This proves the existence
of a root on the interval $\bigl(\ve,\frac3{2\sqrt2}\ve\bigr).$

To check uniqueness, we differentiate $g^-$ with respect to $\delta.$
$$
\frac{\partial g^-}{\partial
\delta}(\delta,\ve)=\delta\left[e^{-\delta+\ve/\sqrt2}-e^{-\rtde}\left(2-e^{-\ve/\sqrt2}\right)\right].
$$
If $\ve<\delta<\frac3{2\sqrt2}\ve,$ then $-\delta+\ve/\sqrt2<-\rtde,$ and we have
$$
\frac{\partial g^-}{\partial \delta}(\delta,\ve)\le\delta
e^{-\rtde}\left(-1+e^{-\ve/\sqrt2}\right)<0.
$$
This completes the proof of Proposition~\ref{pr6-} and Lemma~\ref{l3d}.
\end{proof}
\subsection{How to find the dyadic Bellman function}
\label{howtobeld} 
For simplicity, we only consider the case of $\bel{B}^d=\bel{B}^{d+}.$ What
prompted us to look for the dyadic Bellman function in the family $B_\delta$ from
\eqref{t12}? Firstly, this family was first developed when solving the formal optimal control
problem from \cite{vol,vol1}, where the space under consideration was the dyadic $\BMO.$
Secondly, and more importantly, the following simple proposition shows that the dyadic
Bellman function is locally concave, something that could not be shown directly in the
continuous case. 
\begin{prop}
\label{pr7}
For any three points $x^-,x^+,x\in\Oe$
such that $x=\frac12(x^-+x^+)$ we have
$$
\bel{B}^d_{\ve}(x)\ge\frac12\bel{B}^d_{\ve}(x^-)+\frac12\bel{B}^d_{\ve}(x^+).
$$
\end{prop}
\begin{proof}
Take a sequence
$\left\{\varphi_n\right\}\in \BMO^d_{\ve}(I_-)\cup \BMO^d_{\ve}(I_+)$ such
that
$$
\av{e^{\varphi_n}}{I_\pm}\longrightarrow \bel{B}^d_{\ve}(x^{\pm})~~~\text{as}~~n\to\infty.
$$
We need to check that $\varphi_n\in \BMO^d_{\ve}(I).$ But
$$
\BMO^d_{\ve}(I)=\left\{\varphi:~\varphi|_{I_-}\in\BMO^d_{\ve}(I_-),\;
\varphi|_{I_+}\in\BMO^d_{\ve}(I_+),\;\av{\varphi^2}I-\av{\varphi}I^2
\le\ve^2\right\}.
$$
Since, by assumption, $x\in\Oe,$ we have
$\av{\varphi^2}I-\av{\varphi}I^2\le\ve^2.$ Then we can pass to the limit in
the identity
$$
\av{e^{\varphi_n}}I=\frac12\av{e^{\varphi_n}}{I_-}+\frac12\av{e^{\varphi_n}}{I_+}
$$
to get
$$
\bel{B}^d_{\ve}(x)\ge\lim \av{e^{\varphi_n}}I=\frac12\bel{B}^d_{\ve}(x^-)
+\frac12\bel{B}^d_{\ve}(x^+),
$$
which completes the proof.
\end{proof}
Observe that the statement of the proposition does not hold in the continuous case. In that
case, we have $\ds\BMO_{\ve}(I)\ne\left\{\varphi:~\varphi|_{I_-}\in\BMO_{\ve}(I_-),\;
\varphi|_{I_+}\in\BMO_{\ve}(I_+),\;\av{\varphi^2}I-\av{\varphi}I^2
\le\ve^2\right\},$ since there are other intervals to consider, those with the left endpoint
in $I_-$ and the right one in $I_+.$

We have just proved that $\bel{B}^d_{\ve}$ is locally concave in $\Oe.$ Furthermore, the
reasoning of \eqref{t18} still works and we conclude that 
$$
\bel{B}^d_{\ve}(x)=\exp\left\{x_1+w(x_2-x_1^2)\right\}
$$ 
for a nonnegative function $w$ such
that $w(0)=0.$ What is more, we expect the corresponding matrix $-d^2\bel{B}^d_\ve$ (assuming
sufficient smoothness) to be degenerate, in order for the supremum to be attained for an
extremal function. But we have already described all functions with these properties. They
are the functions $B_{\delta}$ from \eqref{t12}. The condition $\delta\ge\ve$ appears because
the function $\bel{B}^d_{\ve}$ has to be defined on $\Oe,$ $\Oe\subset\Od$ for
$\delta\ge\ve,$ and $\Od$ is just the domain of $B_\delta.$ Thus we look for $\bel{B}^d$
within that family.
\subsection{How to find the dyadic extremal function}
\label{howtoextd} 
Again, we consider only the ``$+$'' case. Recall that in the continuous
case we were looking for a function that would produce equality on every step
in~\eqref{t17d}, i.e. in the Bellman induction of Lemma~\ref{l2d}. Thus, such a function was found
by analyzing what it took to make $B_\delta$ behave as a linear function, that is to have
$$
B_\delta(\alpha_-x^-+\alpha_+x^+)=\alpha_-B_\delta(x^-)+\alpha_+B_\delta(x^+).
$$
We now employ similar reasoning. Namely, we construct the dyadic extremal function for a
point on the top boundary so that we have equality in Lemma~\ref{l3d}, i.e.
\eq[t100]{
B_{\delta(\ve)}\left(\frac12x^-+\frac12x^+\right)=
\frac12B_{\delta(\ve)}(x^-)+\frac12B_{\delta(\ve)}(x^+)
}
at every dyadic split $I=I_-\cup
I_+.$ We construct a function $\varphi_0$ on $I=[0,1]$ for the point $x=(0,\ve^2).$ Then the
function $\varphi_a,$ $\varphi_a(t)=\varphi_0(t)+a,$ is an extremal function for the point
$(a, a^2+\ve^2).$ The proof of Lemma~\ref{l3d} gives us a hint for our construction: the extremum in
\eqref{t39}, \eqref{t40} was realized by a line segment whose center and one of the endpoints
(say $x^-$) lay on the top boundary curve of $\Oe,$ $x_2=x_1^2+\ve^2,$ i.e. $x=(0,\ve^2)$
and $x^-=(a, a^2+\ve^2)$, while the other endpoint, $x^+,$ lay on the bottom boundary curve
$x_2=x_1^2,$ i.e. $x^+=(-a, a^2)$. From the condition $x=\frac12 x_-+\frac12 x_+$ we get
$a=\ve/\sqrt2.$ Only constant functions correspond to the points of the bottom boundary,
so we have to put $\varphi_0(t)=x^+_1=-a$ for $\frac12<t<1$ and on $I^-$ we have to take the
scaled function $\varphi_a$: $\varphi_0(t)=\varphi_a(2t)= \varphi_0(2t)+a$ for $0<t<\frac12.$
The latter relation determines the function $\varphi_0$ recursively: $\varphi_0(t)=(n-1)a$
for $2^{-n-1}<t<2^{-n}.$ This yields the function on Figure~\ref{fi}.

We now describe how to construct an extremal function $\varphi$ when $(x_1,x_2)\ne(0,\ve^2).$
If $x_2=x_1^2+\ve^2,$ i.e. $x$ is on the top boundary, we simply let $\varphi=\varphi_0+x_1$
to get the desired result. Likewise, if $x$ is on the bottom boundary, we let $\varphi=x_1,$
i.e. set the function to be constant on the whole interval. What should we do if $x$ is in
the interior of $\Oe?$ We present two different perspectives on how this situation can be
dealt with. Both lead to the same expression for the extremal function $\varphi.$
\medskip

\noindent{\bf Perspective 1.}
Let us forget for a moment that we are to construct a {\it dyadic} extremal function; then we
can split $I$ so that $x^+$ is on the bottom boundary and $x^-$ is on the top one. Let
$\alpha$ be the splitting parameter, i.e. we have $I_-=(0,\alpha),$ $I_+=(\alpha,1),$ and
$x=\alpha x^-+(1-\alpha)x^+.$ We would like to choose the splitting so that
$$
B_{\delta(\ve)}(x)=\alpha B_{\delta(\ve)}(x^-)+(1-\alpha)B_{\delta(\ve)}(x^+).
$$
Then we can set $\varphi$ to be constant on the right subinterval and the appropriately
scaled function $\varphi_0$ on the left one and apply \eqref{t17d} from Lemma~\ref{l2d} to $I_-$ and
$I_+$ separately. To do this, we place $x^-,$ $x,$ and $x^+$ on a line $\omega^+_\delta$
tangent to the curve $x_2=x_1^2+\delta^2,$ since, according to section \ref{howtoextc},
$B^+_\delta$ is a linear function along any such segment. More precisely, we consider the
line through $x$ that is tangent to $x_2=x_1^2+\delta^2$ and set $x^-$ to be the point of
intersection of the line and the curve $x_2=x_1^2+\ve^2$  and $x^+$ to be the point of
intersection of the line and the curve $x_2=x_1^2.$ Let us calculate $\alpha.$ To avoid
confusion, we will temporarily use $x^0$ when referring to the ``midpoint'' of our segment.
Let us recall the notation of Lemma~\ref{l1d}
\eq[ppp]{
r_1=\rtde;~~r_2=\sqrt{\delta^2-x^0_2+(x^0_1)^2};~~\beta=r_2-r_1;~~\gamma=r_2-\delta;~~
\alpha=\frac{\delta-r_2}{\delta-r_1}. }
Also let
\eq[ppp1]{
\beta_1=\beta+x_1^0;~~\gamma_1=\gamma+x_1^0. }
According to \eqref{t26.1}, the line
$\omega^+_\delta(c)$ tangent to $x_2=x_1^2+\delta^2$ at the point $(c,c^2+\delta^2)$ has the
equation \eq[t200]{ x_2=2cx_1+\delta^2-c^2. } We calculate $c$ using the fact that this line
passes through $x^0.$ Since, in our geometry, $c\ge x^0_1,$ we have $c=x_1^0+r_2.$ Then
\eqref{t200} becomes
$$
x_2=2(x_1^0+r_2)x_1+\delta^2-(x_1^0+r_2)^2
$$
or, equivalently,
$$
(x_1-(x_1^0+r_2))^2=\delta^2+x^2_1-x_2.
$$
This line intersects the top boundary curve at the point $x^-=(\beta_1,\beta_1^2+\ve^2)$ (where we have used the fact that $x^0_1\le x^-_1\le c$); the
intersection with the bottom curve is at $x^+=(\gamma_1,\gamma_1^2).$ The (horizontal) length of the segment $[x^+,x^-]$ is
$\delta-r_1,$ that of the segment $[x^+,x^0]=\delta-r_2,$ so we get $x^0=\alpha x^-+(1-\alpha)x^+.$ Putting everything
together, we obtain the function $\tilde{\varphi}_+$ from the proof of Lemma~\ref{l1d}
$$
\tilde{\varphi}_+(t)=x_1^0+
\begin{cases}
\varphi_0\left(\frac{t}{\alpha}\right)+\beta & \mbox{for}~0<t<\alpha\\
\gamma & \mbox{for}~\alpha<t<1.
\end{cases}
$$
We must pay the price for ignoring the fact that $(0,\alpha)$ is not, in general, a dyadic
interval and, therefore, $\tilde{\varphi}_+$ is not in $\BMO^d_\ve(I).$ How to construct an
appropriate rearrangement $\varphi_+$ of $\tilde{\varphi}_+$ is detailed in the proof of
Lemma~\ref{l1d}.
\medskip

\noindent{\bf Perspective 2.} It is useful to consider another perspective on constructing an
extremal function. We will start with the function $\varphi_0$ built for the point
$x=(0,\ve^2)$ and arrive at the same function $\varphi_+$ for an arbitrary point $x^0$ as the
one in Lemma~\ref{l1d} but using a different reasoning and skipping the $\tilde{\varphi}_+$ phase
altogether. The main feature of this construction is that on every step we define our
function on a dyadic subinterval of $(0,1),$ as opposed to choosing an $\alpha$ and then
approximating it dyadically as in Perspective~1.

Here is the simple logic: Starting with $I=(0,1),$ we will define our function on the right
half of $I,$ then redefine $I$ to be the other half and repeat the procedure. Consider, as
before, the line through $x^0$ tangent to $x_2=x_1^2+\delta^2;$ let $x^t$ and $x^b$ be the
points of intersection of the tangent with the top and bottom boundary of $\Oe,$
respectively. If $x^0$ is closer to $x^t$ than to $x^b,$ set $\varphi$ to be the
appropriately scaled (and adjusted to have the prescribed average) function $\varphi_0$ on
$I_+$ and replace $x_0$ with $2x_0-x^t.$ If, on the contrary, $x^0$ is closer to $x^b$ than
to $x^t,$ set $\varphi$ to be the appropriately chosen constant on $I_+$ and replace $x_0$
with $2x^0-x^b.$ In either case, replace $I$ with $I_-$ and repeat. If $x^0$ is exactly in
the middle between $x^b$ and $x^t,$ let $\varphi$ be the scaled $\varphi_0$ on $I_+$ and
constant on $I_-;$ stop.

We will now make this procedure more precise and show why the function so obtained is the
same as the one used to prove Lemma~\ref{l1d}.

Start with a point $x^0\in\Oe.$ Let $x^*=x^0,$ $I=(0,1)$ (the initial settings; $x^*$ and $I$
will be redefined in the procedure). Let $r_1,$ $\beta_1,$ and $\gamma_1$ be defined by
\eqref{ppp} and \eqref{ppp1} (these will not be redefined).
\ben
\item
\label{s1}
Let
$r_2=\sqrt{\delta^2-x^*_2+(x^*_1)^2}$

-- if $\delta+r_1<2r_2,$ go to Step \ref{s2};

-- if $\delta+r_1>2r_2,$ go to Step \ref{s3};

-- if $\delta+r_1=2r_2,$ go to Step \ref{s4}.
\item
\label{s2}
Let $\varphi|_{I_+}=\gamma_1.$
\item
Let $x^b=(\gamma_1,\gamma_1^2),~x^*:=2x^*-x^b,~I:=I_-.$ Go to Step \ref{s1}.
\item
\label{s3}
Let $\varphi|_{I_+}=\varphi_0(2^kt+1)+\beta_1.$
\item
Let
$x^t=(\beta_1,\beta_1^2+\ve^2),~x^*:=2x^*-x^t,~I:=I_-.$ Go to Step \ref{s1}.
\item
\label{s4}
Let $\varphi|_{I_+}=\varphi_0(2^kt+1)+\beta_1,~\varphi|_{I_-}=\gamma_1.$ Stop.
\een
Since on
every run of the loop we define $\varphi$ on half of the current interval $I$ and then rename
the other half $I,$ at the end we have defined $\varphi$ almost everywhere on $(0,1).$
Furthermore, since every interval in the process is dyadic and $\varphi\in\BMO^d_\ve(J)$ for
every interval $J$ that turns up on step \ref{s2}, \ref{s3}, or \ref{s4}, we conclude that
$\varphi\in \BMO^d_\ve([0,1])$ (see the short discussion after the proof of Proposition~\ref{pr6}).
All the action happens on the same line tangent to the parabola $x_2=x_1^2+\delta^2,$
guaranteeing equality in \eqref{t17d} of Lemma~\ref{l2d}.

The inequality $\delta+r_1<2r_2$ (or $>,=$) is equivalent to the inequality
$\delta-r_2<r_2-r_1$ (or $>,=$), the statement that the distance from $x^0$ to the bottom
boundary curve is less than that to the top one. Alternatively, this inequality is equivalent
to $\frac{\delta-r_2}{\delta-r_1}<\frac12,$ i.e., in the language of Perspective~1,
$\alpha<\frac12.$ But comparing this, current $\alpha$ to $1/2$ is the same as determining
whether the current dyadic digit of the original $\alpha$ is $0$ or $1.$ Indeed, if the
current $x^*$ is closer to the top boundary, its next value will be twice as far from it;
same holds for the bottom boundary. Let us quantify this.

Let $z_0=\alpha,$ $z_k=\frac{\delta-(r_2)_k}{\delta-r_1},$ the value on the $k\!$-th step of
our procedure. By construction, if $z_{k-1}>1/2,$ then $z_k=2z_{k-1}-1,$ and if
$z_{k-1}<1/2,$ then $z_k=2z_{k-1}.$ Thus $z_k=\{2z_{k-1}\},$ the fractional part of
$2z_{k-1}.$ Then $\alpha_k\df[2z_{k-1}]$ (the integer part) is the $k\!$-th dyadic digit of
$\alpha.$ Recalling definition \eqref{7.1}, we see that the function $\varphi$ so obtained is
indeed the same as $\varphi_+$ in Perspective~1.
\section{Conclusion}
In this section, we summarize what has been achieved, specify which obstacles need to be
overcome on the way to generalizing the results, and outline immediate and long-term
prospects.

From a purely practical viewpoint, we have obtained sharp new results in a widely-used
inequality; in addition, the dyadic $\BMO$ formulation is common in applications, therefore
exploring the problem in this setting --- and showing that the results differ significantly
from the continuous setting --- is important.

Equally significant is the methodological aspect of this work. We have added another
nontrivial example to the short list of explicit Bellman functions. This paper can be viewed
as an excellent case study, following every step in the recent explicit-Bellman paradigm. As
far as we know, our transition to the dyadic case from the continuous one is unique in
literature; as mentioned in the introduction, the usual way is the opposite. The dyadic
setting has been prevalent in Bellman function studies, our getting of an explicit
continuous-case Bellman function is noteworthy in itself.

There are several natural questions one may ask:
\medskip

1. Can the results be extended to the $L^p\!$-based $\BMO?$
\medskip

The choice of variables \eqref{var} (and so the associated Bellman function definitions)
depends heavily on the $L^2\!$-structure of our $\BMO.$ For $p>1$ it is possible to consider
the choice $x_2=\ave{\varphi^p},$ although the associated norms are not the regular
$L^p\!$-based $\BMO$ norms. It appears that an altogether different Bellman setup may be
needed for the $L^1$ case.
\medskip

2. Can the results be extended to higher dimensions?
\medskip

Once we move to higher dimensions, there is the question of how one defines $\BMO.$ Typical
definitions are using cubes or balls, although other definitions are possible. Since our
technique depends critically on one's ability to split a body in $\mathbb{R}^n$ into bodies
of the same type, it seems that the dyadic case is more amenable to higher-dimensional
considerations because in the dyadic situation we have no problem splitting a cube into a
union of smaller cubes. In the continuous case, however, the crucial splitting tool we have
used, Lemma~\ref{l4c}, is pointedly one-dimensional. We could easily generalize our results to the
$n\!$-parameter $\BMO$ on rectangles, but this appears to be of little interest.

Often in Bellman proofs one relies on a certain dyadic Bellman function to handle all
dimensions. Naturally, our continuous-to-dyadic way of solving the problem does not go
through in that sense. In addition, the continuous and dyadic results are expected to be
increasingly different as dimension grows. Overall, new techniques are needed (work is
underway) to deal with the higher-dimensional case.

Despite our present inability to handle the multidimensional case, we would like to put
forward two related conjectures, for the $\BMO$ defined on cubes.

{\bf Conjecture~1.} Theorems~\ref{t1c} and~\ref{t2c} remain true in the multidimensional case, i.~e. in the
non-dyadic case the Bellman function does not depend on the dimension.

{\bf Conjecture~2.} In the dyadic $n\!$-dimensional case the Bellman functions are
$B^\pm_{\delta_n^\pm},$ where the parameters $\delta_n^\pm=\delta_n^\pm(\ve)$ are the
solutions of the following equations
$$
(1\mp\rtde)\exp(\pm\rtde\mp\delta)\left(2^n-e^{\pm(2^{n/2}-2^{-n/2})\ve}\right)=
(1\mp\delta)(2^n-1)e^{\mp\ve2^{-n/2}},
$$
and, therefore, the corresponding constants $C^d_n(\ve)$ and $\ve^d_0(n)$ are
\begin{gather*}
C^d_n(\ve)=\frac{(2^n-1)e^{-\ve2^{-n/2}}}{2^n-e^{(2^{n/2}-2^{-n/2})\ve}},
\\
\ve^d_0(n)=\frac{n\log2}{2^{n/2}-2^{-n/2}}.
\end{gather*}
These conjectures are true if it is true that the extremal function corresponding the point
$(0,\ve^2)$ is
$$
\varphi_0(t_1,\ldots,t_n)=\ve\left(n\log\frac1{\;\max t_k}-1\right)
$$
in the non-dyadic case and
$$
\varphi_0^d(t_1,\ldots,t_n)=-\ve2^{-n/2}+
\sum_{k=1}^\infty(2^{n/2}-2^{-n/2})\ve\chi_{_{[0,2^{-k}]}}(\max t_k)
$$
in the dyadic one.
\medskip

3. Can the classical weak-from John--Nirenberg inequality be handled by the methods of this paper?
\medskip

At the moment, this appears to be the most promising of all directions of further research on
the topic. By design, the Bellman function for a distributional inequality will have one more
variable (at least, another parameter), compared to the integral case. This implies that the
order of Bellman PDE in the weak-form case will be higher.

On the other hand, we have a ready choice of variables, just reusing the ones form this
paper. The usual logic that allows one to establish a finite-difference inequality for the
Bellman function still works. In~\cite{tao} a Bellman-type function satisfying this
inequality (so called supersolution) was found for the dyadic BMO. This showed that the
Bellman function method works for the weak form of the John--Nirenberg inequality. However,
not being the true Bellman function, this supersolution only gives suboptimal (not sharp)
constants in the inequality. It is our hope to be able to find the true Bellman function for
the this inequality as well. Being the averages of functions, our variables have a clear
martingale structure, thus we expect to be able to rewrite that inequality as a homogeneous
Monge-Amp\'{e}re equation, just as we have done here. Though that equation will not reduce to
an ODE, there has been a recent surge (and success) in in-depth studies of the connection of
such PDEs with the Bellman function method. All of this gives this problem a very promising
outlook.

\end{document}